\documentclass[reqno, 12pt]{amsart}

\usepackage{amsmath, amsthm, amssymb, paralist, cite, color}
\usepackage[all]{xypic}
\definecolor{e-mail}{rgb}{0,.40,.80}
\definecolor{reference}{rgb}{.20,.60,.22}
\definecolor{citation}{rgb}{0,.40,.80}

\usepackage[pdfstartview=FitH, colorlinks, linkcolor=reference, citecolor=citation, urlcolor=e-mail]{hyperref}

\newtheorem{thm}{Theorem}
\newtheorem{cor}[thm]{Corollary}
\newtheorem{lem}[thm]{Lemma}
\newtheorem{prop}[thm]{Proposition}

\theoremstyle{definition}
\newtheorem{defn}[thm]{Definition}
\theoremstyle{remark}

\numberwithin{thm}{section}
\theoremstyle{definition}

\theoremstyle{definition}

\newtheorem{eg}[thm]{Example}
\theoremstyle{definition}
\numberwithin{equation}{section}

\def\CX{{\mathbb C}}

\def\ZX{{\mathbb Z}}

\def\GL{{\rm GL}}
\def\SL{{\rm SL}}
\def\gl{{\mathfrak{gl}}}

\def\d{{\delta}}

\def\calN{{\mathcal N}}

\def\Scal{{\mathrm{Scal}}}

\def\sd{{\sigma\delta}}
\def\SD{{\sigma\d}}

\def\dd{{\partial}}

 \title[Galois groups and integrability of difference equations]{Galois groups for integrable and projectively integrable linear difference equations}
 
\author[C.E. Arreche]{Carlos E. Arreche}
\address{Mathematics Department, North Carolina State University, Raleigh, NC 27695}
\email{cearrech@math.ncsu.edu}
\thanks{The work of the first author was partially supported by an NSF Alliance Postdoctoral Fellowship and NSF grant CCF-0952591.}

\author[M.F. Singer]{Michael F. Singer}
\address{Mathematics Department, North Carolina State University, Raleigh, NC 27695}
\email{singer@math.ncsu.edu}
\thanks{The work of the  second author was partially supported by a grant from the Simons Foundation (\#349357, Michael Singer).}

\begin{document}

\begin{abstract}
We consider first-order linear difference systems over $\mathbb{C}(x)$, with respect to a difference operator $\sigma$ that is either a shift $\sigma:x\mapsto x+1$, $q$-dilation $\sigma:x\mapsto qx$ with $q\in{\mathbb{C}^\times}$ not a root of unity, or Mahler operator $\sigma:x\mapsto x^q$ with $q\in\mathbb{Z}_{\geq 2}$. Such a system is integrable if its solutions also satisfy a linear differential system; it is projectively integrable if it becomes integrable ``after moding out by scalars." We apply recent results of Sch\"{a}fke and Singer to characterize which groups can occur as Galois groups of integrable or projectively integrable linear difference systems. In particular, such groups must be solvable. Finally, we give hypertranscendence criteria.
\end{abstract}

\maketitle

\section{Introduction} \label{intro-sec}   In \cite{HaSi08} a differential Galois theory of linear difference equations was developed as a tool to understand the differential properties of solutions of linear difference equations. This theory associates to a system of linear difference equations $Y(\sigma(x)) = A(x) Y(x)$ a  {\it differential Galois group}. This is a linear differential algebraic group, that is a group of matrices whose entries are functions satisfying a fixed set of (not necessarily linear) differential equations. Differential properties of solutions of the linear difference equation are measured by group-theoretic properties of the associated differential Galois group. For example, a group-theoretic proof is given in \cite{HaSi08} of H\"older's Theorem that the Gamma Function satisfies no polynomial differential equation, that is, the Gamma Function is hypertranscendental. In general, one can measure the amount of differential dependence among the entries of a fundamental solution matrix of $Y(\sigma(x)) = A(x) Y(x)$ by the size of its associated {differential Galois group:} the larger the group, the fewer the differential relations {that} hold among these entries. 

This theme has been taken up in \cite{DHR15}, where the authors develop criteria that allow them to show that the generating series $F(x) = \sum_{n\in \calN} x^n$ of certain $p$-automatic sets $\calN$ are hypertranscendental. It is known \cite{B94} that  these generating series satisfy Mahler equations,
 that is, equations of the form 
 \[\sigma^m(F(x)) + a_{m-1}(x)\sigma^{m-1}(F(x))+ \ldots + a_0(x) F(x) = 0,\]
 where $\sigma(x) = x^p$ and the $a_i(x) \in \CX(x)$.  Dreyfus, Hardouin, and Roques {develop criteria in \cite{DHR15}} to ensure that  certain Mahler equations have $\SL_n$ or $\GL_n$ as  their associated differential Galois groups. Using these criteria they show, for example, that the generating series associated with the Baum-Sweet  and the Rudin-Shapiro sequences are hypertranscendental. In \cite{DHR16}, the authors develop similar criteria to ensure that the differential Galois groups of  certain $q$-difference equations are large, and apply these {criteria} to show that certain generalized $q$-hypergeometric series are hypertranscendental.

{They prove} the validity of their criteria {by appealing to} B\'ezivin's result \cite{Bez94}, which states that a power series that simultaneously satisfies a Mahler equation and a linear differential equation must be a rational function, and {to} Ramis's result \cite{Ramis92}, which states that a power series that simultaneously satisfies a $q$-difference equation and a linear differential equation must also be a rational function. {These classical results of B\'ezivin and Ramis are reproved and generalized by Sch\"afke and Singer in \cite{SS16}, where they also} classify series and other functions that simultaneously satisfy either a shift-difference equation and a linear differential equation, or {a pair} of difference equations for suitable pairs of operators.  The authors of \cite{SS16} prove these results by considering systems of linear differential and difference equations 
\begin{eqnarray*}
\dd Y(x) =A(x)Y(x), \ \ \ \ \ 
\sigma Y(x) =B(x)Y(x),
\end{eqnarray*}
 where $\dd = \frac{d}{dx}$  and $\sigma$  is a shift operator {$\sigma(x) = x+a$}, 
$q$-{dilation} operator {$\sigma(x) = qx$}, or Mahler operator {$\sigma(x) = x^p$}.

Here $A(x)$ and $B(x)$ are $n\times n$ matrices with rational function entries.
Assuming a consistency hypothesis, they show in \cite[Thm.~1]{SS16} that such a system is equivalent to a system of a very simple form and deduce, among other conclusions, {their generalizations of B\'ezivin's and Ramis's results}. 

In this paper we consider systems of difference equations
\begin{equation}\label{introeq}\sigma(Y(x)) = A(x) Y(x), \end{equation}
where $A(x) \in \GL_n(C_0(x))$, $C_0$ is an algebraically closed field, and $\sigma$ is a shift operator, $q$-dilation operator, or Mahler operator. Using Theorem 1 of \cite{SS16}, we characterize those groups  that occur as differential Galois groups under  the assumption that (\ref{introeq}) is integrable (i.e., its solutions also satisfy a linear differential system)  or projectively integrable\footnote{Integrable/projectively integrable are also called isomonodromic/projectively isomonodromic in the literature, e.g., \cite{DHR15,DHR16}} (i.e., it becomes integrable after ``moding out by scalars''; for precise definitions see  Definitions~\ref{intdef} and ~\ref{projintdef}).  In particular, we show in the integrable case that the Galois groups are abelian, and in the projectively integrable case that these Galois groups are finite abelian extensions of abelian groups.

Using these characterizations, we reprove and extend the criteria of \cite{DHR15} and \cite{DHR16} when $\sigma$ is a Mahler or $q$-dilation operator, respectively.  We also derive  corresponding criteria when $\sigma$ is a shift operator, and apply  them to show  that the solutions of certain shift-difference equations are hypertranscendental\footnote{In \cite{DHR16} the authors also consider $q'$-difference relations among solutions of a $q$-difference equation, where $q$ and $q'$ are sufficiently independent complex numbers.  Using the Galois theory developed in \cite{OW15}, they  develop Galois-theoretic criteria  to deduce $\sigma_{q'}$-transcendence results. Difference systems with such pairs of operators were also considered in \cite{SS16}. Theorem 12 of this latter paper  should yield new proofs and extensions of some of these results, but we do not consider this case or other pairs of difference operators in this paper.}. Besides  applications to questions of hypertranscendence,  our results are  also useful in  understanding the inverse problem for the differential Galois theory of difference equations  by showing that certain groups cannot occur  as differential Galois groups. These results have also been  applied in \cite{Arreche15b} to develop algorithms that compute differential Galois groups of difference equations. 

The rest of the paper is organized as follows. A key idea of this paper is that the results of \cite{SS16} allow one to reduce questions concerning the differential Galois groups of integrable and  projectively integrable  equations (\ref{introeq}) to questions concerning difference systems with constant coefficients. In Section~\ref{constant-sec} we discuss systems with constant coefficients and their Galois groups. In Section~\ref{int-sec} we characterize the differential Galois groups of integrable systems (\ref{introeq}), and in Section~\ref{proj-sec} we do the same in the projectively integrable case.  In Section~\ref{hyper-sec} we apply our results to questions of hypertranscendence  and compare our results with those of \cite{DHR15} and \cite{DHR16}.

\section{Constant systems} \label{constant-sec}

Before we discuss the differential Galois theory of difference equations, we recall the Galois theory of linear difference equations \cite{PuSi}\footnote{More general theories exist (e.g., \cite{wibmer10})  but the theory in \cite{PuSi} suffices for our purposes.}. Let $(k,\sigma)$ be a difference field, that is, a field\footnote{All fields considered in this paper are of characteristic zero.} $k$ with an automorphism $\sigma$, and assume that $k^\sigma = \{c \in k \ | \ \sigma(c) = c\}$ is algebraically closed. This theory associates to an equation $\sigma(Y) = AY$, where $A \in \GL_n(k)$, a $\sigma$-simple ring $S = k[Z,\frac{1}{\det Z}]$, called the {\it $\sigma$-Picard-Vessiot ring, which} is generated by the entries of a fundamental solution matrix $Z$ and the inverse of its determinant. This ring is unique up to $\sigma$-$k$-isomorphism and has no nilpotent elements but may have zero divisors. Its total ring of quotients $K$ is called the {\it total $\sigma$-Picard-Vessiot ring}. It has the property that the subring $K^\sigma$ of elements left fixed by $\sigma$ coincides with the corresponding subring $k^\sigma$\ of $k$. The group $G$ of $\sigma$-$k$-algebra automorphisms of this ring is called the {\it $\sigma$-Galois group} and is a linear algebraic group defined over $k^\sigma$. There is a Galois correspondence between linear algebraic subgroups of $G$ and $\sigma$-subrings of $K$, containing $k$, all of whose nonzero divisors are invertible.

We begin by proving some general results about constant difference equations and their $\sigma$-Galois groups, before applying these results to our main cases of interest in Proposition~\ref{sigmagp}.

\begin{lem}\label{lemcyclic} Let $C$ be an algebraically closed field and $\sigma:C\rightarrow C$ be the identity map. Suppose that $A\in\GL_n(C)$.
\begin{enumerate}
\item  The $\sigma$-Galois group of $\sigma(Y)=AY$ over $C$ is the Zariski closure of the group generated by $A$.
\item If $k$ is any $\sigma$-field with $k^\sigma=C$, the $\sigma$-Galois group of $\sigma(Y)=AY$ over $k$ is a closed subgroup of the Zariski closure of the group generated by $A$.\end{enumerate} \end{lem}

\begin{proof}{(1)}~Let $E$ be the total $\sigma$-Picard-Vessiot ring of $\sigma(Y)=AY$ over $C$, and let $G$ be its $\sigma$-Galois group. Note that $\sigma$ induces a difference automorphism of $E$ over $C$ leaving the elements of $C$ fixed. Therefore $\sigma \in G$. If $H$ is the Zariski closure of the group generated by $\sigma$, then $H \subset G$.  The fixed ring of $H$ is $C$, so by the Galois correspondence we have $H=G$.

{(2)}~Let $F$ be the total $\sigma$-Picard-Vessiot ring of $\sigma (Y) = AY$ over $k$ and $Z\in \mathrm{GL}_n(F)$ be a fundamental solution matrix of this equation. Let $S = C[Z,\frac{1}{\det Z}]$.  By \cite[Lem.~19 and Cor.~20]{FSW10}, we can embed the total quotient ring $E$ of $S$ into $F$. By \cite[Prop.~1.23]{PuSi}, $E$ is the total $\sigma$-Picard-Vessiot ring of $\sigma (Y) = AY$ over $C$.  Let $G$ be the $\sigma$-Galois group of $F$ over $k$ and $H$ be the $\sigma$-Galois group of $E$ over $C$. The map sending $\gamma \in G$ to its restriction $\gamma|_E$ to $E$ is an isomorphism of $G$ onto the $\sigma$-Galois group of $E$ over $E\cap k$, which latter group is a subgroup of $H$. The result now follows from~{(1)}.\end{proof}
 
In order to apply Lemma~\ref{lemcyclic}, we will need to know which groups occur as Zariski closures of cyclic groups. This information is given in the next lemma. We denote by $\mathbb{G}_a(C)$ the additive group of $C$ and by $\mathbb{G}_m(C)$ the multiplicative group of nonzero elements of $C$.

\begin{lem}\label{lemcyclic2} Suppose that $A\in \GL_n(C)$.
\begin{enumerate}

\item A linear algebraic group $G$ is the Zariski closure of the group generated by $A$ if and only if $G$ is isomorphic to a group of the form 
\[\mathbb{G}_m(C)^r \times \mathbb{G}_a(C)^s \times \mathbb{Z}/t\mathbb{Z},\]
where $r$ and $s$ are nonnegative integers, $s \leq 1$, and $t$ is a positive integer.

\item Suppose furthermore that $A$ is {\em diagonalizable}. A linear algebraic group $G$ is the Zariski closure of the group generated by $A$ if and only if $G$ is isomorphic to a group of the form 
\[\mathbb{G}_m(C)^r  \times \mathbb{Z}/t\mathbb{Z},\]
with $r$ and $t$ as above.
\end{enumerate}
\end{lem}

\begin{proof}(1)~Let $G$ be the Zariski closure of the group generated by $A$. From \cite[\S15.5]{humphreys}, we may write $A = A_sA_u$, {where} $A_s, A_u \in G$, $A_s$ is semisimple, $A_u$ is unipotent, and $A_sA_u=A_uA_s$. Let $G_s$ be the Zariski closure of the group generated by $A_s$, and let $G_u$ be the Zariski closure of the group generated by $A_u$.  We have that $G_s \cap G_u = \{id\}$ and $A \in G_sG_u$, so $G \simeq G_s\times G_u$. By \cite[Lem.~C of \S15.1]{humphreys}, $G_u \simeq\mathbb{G}_a(C)$. Since $A_s$ is diagonalizable, the group $G_s$ is diagonalizable and therefore isomorphic to a group of the form $\mathbb{G}_m(C)^r  \times H$, where $H$ is a finite group\cite[\S16.2]{humphreys}. Let $\pi:\mathbb{G}_m(C)^r  \times H \rightarrow H$ be the canonical projection. The image of $A_s$ generates a Zariski-dense subgroup of $H$, so $H$ is cyclic.

Assume $G$ is isomorphic to $\mathbb{G}_m(C)^r \times \mathbb{G}_a(C)^s \times \mathbb{Z}/t\mathbb{Z}$. Let $a= (a_1, \ldots , a_r) \in \mathbb{G}_m(C)^r$,  where $a_1, \ldots , a_r$ are {\it multiplicatively independent}, that is, they satisfy the property: if $\prod_{i=1}^r a_i^{m_i} = 1$ for some $ m_1, \ldots , m_r \in \mathbb{Z}$, then $m_1=\ldots = m_r = 0$. If $H$ is a proper subgroup of $\mathbb{G}_m(C)^r$ then the elements of $H$ satisfy a relation of the form $\prod_{i=1}^r a^{m_i} = 1$ \cite[\S16.1]{humphreys}, so $a$ generates a Zariski-dense subgroup of $\mathbb{G}_m(C)^r$. If $s=1$, let $id \neq b \in \mathbb{G}_a$. We have that the Zariski closure of the group generated by $b$ equals $\mathbb{G}_a(C) $ \cite[\S15.1]{humphreys}. Let $c$ be a generator of $\mathbb{Z}/t\mathbb{Z}$. The element $y = (a,b,c)$ generates a Zariski-dense subgroup of $G$. To see this, let $H$ be the Zariski closure of the group generated by $y$. The elements $(a,id,c)$ and $(id, b, id)$ lie in $H$ so 
$\mathbb{G}_m(C)^r \times  \mathbb{Z}/t\mathbb{Z}\subset H$ and $\mathbb{G}_a(C)^s \subset H$.

(2)~This follows from the proof of~{(1)}.\end{proof}

We will now consider three concrete examples of difference fields $(k, \sigma)$. In each of these cases, we have that $k^\sigma=C$, which we always assume to be algebraically closed.\vspace{.1in}

\begin{itemize}
\item[{\rm case s:}] The field $k:=C(x)$ and $\sigma$ is the shift operator defined by $\sigma(x)=x+1$.
\item[{\rm case q:}] The field $k:=C(x)$ and $\sigma$ is the $q$-dilation operator defined by $\sigma(x)=qx$ for some $q\in C$, $q\neq0$ and not a root of unity.
\item[{\rm case m:}] The field $k:=C(\mathrm{log}(x),\{x^{1/\ell}\ | \ \ell\in\mathbb{N}\})$ and $\sigma$ is the Mahler operator defined by $\sigma(x^\alpha)=x^{q\alpha}$ and $\sigma(\mathrm{log}(x))=q\mathrm{log}(x)$ for some integer $q\geq2$.
\end{itemize}

\vspace{.1in}

The following result characterizes which linear algebraic groups occur as $\sigma$-Galois groups for constant linear difference equations over the difference fields $(k,\sigma)$ corresponding to cases s, q, and m.

\begin{prop} \label{sigmagp}Suppose that $A\in\GL_n(C)$.
\begin{enumerate} 
\item A linear algebraic group $G$ is a $\sigma$-Galois group for $\sigma(Y) = AY$ over $C$ if and only if it is isomorphic to a group of the form
\[\mathbb{G}_m(C)^r \times \mathbb{G}_a(C)^s \times \mathbb{Z}/t\mathbb{Z},\]
where $r$ and $s$ are nonnegative integers, $s \leq 1$, and $t$ is a positive integer.

\item In case s, a linear algebraic group $G$ is a $\sigma$-Galois group for $\sigma(Y) = AY$ over $k$ if and only if it is isomorphic to a group of the form
\[\mathbb{G}_m(C)^r \times \mathbb{Z}/t\mathbb{Z}, \]
where $r$ is a nonnegative integer and $t$ is a positive integer.

\item In case q, a linear algebraic group $G$ is a $\sigma$-Galois group for $\sigma(Y) = AY$ over $k$ if and only if it is isomorphic to a group of the form
\[\mathbb{G}_m(C)^r \times \mathbb{G}_a(C)^s \times \mathbb{Z}/t\mathbb{Z} \times \mathbb{Z}/p\mathbb{Z},\]
where $r$ and $s$ are nonnegative integers, $s \leq 1$, and $t$ and $p$  are positive integers.

\item In case m, a linear algebraic group $G$ is a $\sigma$-Galois group for $\sigma(Y) = AY$ over $k$ if and only if it is isomorphic to a group of the form
\[\mathbb{G}_m(C)^r \times \mathbb{G}_a(C)^s \times \mathbb{Z}/t\mathbb{Z} \times \mathbb{Z}/p\mathbb{Z},\]
where $r$ and $s$ are nonnegative integers, $s \leq 1$, and $t$ and $p$  are positive integers.

\end{enumerate}
\end{prop}

\begin{proof} (1)~follows directly from Lemmas~\ref{lemcyclic}(1) and \ref{lemcyclic2}(1).

(2)~Let  $G$ be isomorphic to  $\mathbb{G}_m(C)^r  \times \mathbb{Z}/t\mathbb{Z}$. Let $  a_1, \ldots, a_r \in C$ be multiplicatively independent, let $A_0:=\bigl(\begin{smallmatrix} 1 & 1 \\ 0 & 1\end{smallmatrix}\bigr)$, and let $A = \mathrm{diag}( a_1, \ldots , a_r, A_0, \zeta)$, where $\zeta$ is a primitive {$t$-{th}} root of unity. From the proof of Lemma~\ref{lemcyclic2} we see that the Zariski closure of the group generated by $A$ is isomorphic to $\mathbb{G}_m(C)^{r} \times \mathbb{G}_a(C) \times \mathbb{Z}/t\mathbb{Z}$, and so the $\sigma$-Galois group of $\sigma(Y) = AY$ over $C$ is this group as well. Let $K$ be the total $\sigma$-Picard-Vessiot ring of $\sigma(Y) = AY$ over $C$. Considering the equation $\sigma(Y) =A_0Y$, one sees that $K$ contains a solution $u$ of the equation $\sigma(u) = u+1$. Therefore $C(u)$ is $\sigma$-isomorphic to $k$. Furthermore, the $\sigma$-Galois group of $K$ over $C(u)$ is isomorphic to $\mathbb{G}_m(C)^r  \times \mathbb{Z}/t\mathbb{Z} $. This yields one implication of~{(2)}.

Given $A \in \mathrm{GL}_n(C)$, let $G$ be the $\sigma$-Galois group of $\sigma(Y) = AY$ over $k$. We begin by showing that this equation is gauge equivalent over $k$ to an equation $\sigma(Y) = \tilde{A}Y$, where $\tilde{A} \in \mathrm{GL}_n(C)$ is diagonalizable. Write  $A=A_sA_u$ with $A_s,A_u \in \mathrm{GL}_n(C)$, where $A_s$ is diagonalizable, $A_u$ is unipotent, and $A_sA_u = A_uA_s$. Let $A_u = \exp(N)$, where $N$ is nilpotent, and let $M := \exp(-x N)$. Note that $M\in\GL_n(k)$ commutes with $A_s$ and $\sigma(M) = MA_u^{-1}$. Making the gauge transformation $V = MY$, we have 

\[\sigma(V)  =  \sigma(M)AM^{-1}V = MA_u^{-1}A_uA_sM^{-1} V= A_sV.\] 

Therefore we may assume that $A$ is diagonalizable. Lemmas~\ref{lemcyclic}(2) and \ref{lemcyclic2}(2) imply that $G$ is a closed subgroup of a group isomorphic to $\mathbb{G}_m(C)^r  \times \mathbb{Z}/t\mathbb{Z}$. The unipotent radical of $G$ must therefore be trivial, so $G$ is isomorphic to a group of the form $\mathbb{G}_m(C)^r\times H$, where $H$ is a finite group. Now \cite[Prop.~1.20]{PuSi} implies that $G/G^0\simeq H$ is a cyclic group, so $G$ is of the desired form.

(3)~Let $G$ be isomorphic to  $\mathbb{G}_m(C)^r \times \mathbb{G}_a(C)^s \times \mathbb{Z}/t\mathbb{Z} \times \mathbb{Z}/p\mathbb{Z}$, and assume that $s=1$ (the case when $s=0$ is similar but easier). Let $  a_1, \ldots, a_r \in C$ be such that $q, a_1, \ldots, a_r$ are multiplicatively independent. Let $A = \mathrm{diag}(q^{1/p}, a_1, \ldots , a_r, A_0, \zeta)$, where $A_0$ is as above and $\zeta$ is a primitive {$t$-{th}} root of unity. From the proof of Lemma~\ref{lemcyclic2} we see that the Zariski closure of the group generated by $A$ is isomorphic to $\mathbb{G}_m(C)^{r+1} \times \mathbb{G}_a(C)^s \times \mathbb{Z}/t\mathbb{Z}$, and so the $\sigma$-Galois group of $\sigma(Y) = AY$ over $C$ is this group as well. Let $K$ be the total $\sigma$-Picard-Vessiot ring of this equation. Then $K$ contains a solution $u$ of the equation $\sigma(u) = q^{1/p}u$, and so $v = u^p$ satisfies $\sigma(v) = q v$.  Therefore $C(v)$ is $\sigma$-isomorphic to $k$. Furthermore, the $\sigma$-Galois group of $K$ over $C(v)$ is isomorphic to $\mathbb{G}_m(C)^r \times \mathbb{G}_a(C)^s \times \mathbb{Z}/t\mathbb{Z} \times \mathbb{Z}/p\mathbb{Z}$. This yields one implication of~{(3)}.

Given $A \in \mathrm{GL}_n(C)$, let $G$ be the $\sigma$-Galois group of $\sigma(Y) = AY$ over $k$. Lemmas~\ref{lemcyclic}(2) and \ref{lemcyclic2}(1) imply that $G$ is a closed subgroup of a group isomorphic to $\mathbb{G}_m(C)^r \times \mathbb{G}_a(C)^s \times \mathbb{Z}/t\mathbb{Z}$. The unipotent radical of $G$ must therefore be a subgroup of $\mathbb{G}_a(C)^s$. Furthermore, since $G$ is abelian, it is isomorphic to a group of the form $G_s\times G_u$, where $G_s$ is a product of copies of $\mathbb{G}_m(C)$ and a finite abelian group $H$, and $G_u$ is a product of copies of $\mathbb{G}_a(C)$. Therefore the identity component $G^0$ of $G$ is of the form $\mathbb{G}_m(C)^r \times \mathbb{G}_a(C)^s$, where $r$ is a nonnegative integer and $s=0$ or $1$. Now \cite[Prop.~12.2.1]{PuSi} implies that $G/G^0\simeq H$ is a finite  abelian group with at most two generators, so $G$ is of the desired form.

(4) We will prove this result by reducing the problem to (3)~above.  Note that the difference field $C(\log(x))$ with $\sigma(\log(x)) = q\log(x)$ is isomorphic as a difference field to $C(x)$ with $\sigma(x) = qx$.  Therefore necessary and sufficient conditions for a group to be a $\sigma$-Galois group of $\sigma(Y) = AY$ are the same for these fields, {since $A\in\GL_n(C)$}.  We will show that the $\sigma$-Galois group for such an equation over $k$ is the same as the $\sigma$-Galois group over $F := C(\log(x))$. 

Let $K$ be the total Picard-Vessiot ring for $\sigma(Y) = AY$ over $k$, and let $L$ be the total Picard-Vessiot ring for the same equation over $F$. Let $G$ be the $\sigma$-Galois group of $K$ over $k$, and let $H$ the $\sigma$-Galois group of $L$ over $F$.  As in the proof of Lemma~\ref{lemcyclic}, the map sending $\gamma \in G$ to its restriction $\gamma|_{L}$ is an isomorphism of $G$ onto the $\sigma$-Galois group of $L$ over $L \cap k$, which is a subgroup of $H$ (note that $L \cap k$ is again a $\sigma$-field). If we can show that $E:= L \cap k = F$, then {it will  follow that} these two $\sigma$-Galois groups {are} the same. 

We claim that $E$ is finitely generated as a field extension of $C$. We have that $L$ is the total quotient ring of a Picard-Vessiot ring $S = F[Z, 1/\det(Z)]$ over $F$. {It follows from \cite[Cor.~1.16]{PuSi} that} $S = S_1 \oplus \ldots \oplus S_t$, where each $S_i$ is a domain and finitely generated over $F$. Therefore $L = R_1\oplus \ldots \oplus R_t$, where each $R_i$ is {a finitely generated field extension of} $F$. Letting $\pi_1:L \rightarrow R_1$ be the projection map, we have that $\pi_1$ is an injection of $E$ into $R_1$. Since a subfield of a finitely generated field is finitely generated, we have that $E$ is finitely generated over $F= C(\log(x))$, and so finitely generated over $C$.

Now assume that $E \neq F$. This implies that for some integer $\ell$, there exist nonzero polynomials $f$  and $g$ with coefficients in $F$, at least one of them with positive degree, such that $f(x^{1/\ell})/ g(x^{1/\ell}) \in E$.  This implies that $x^{1/\ell}$ is algebraic over $ E$ of some positive degree, say $N$. Applying $\sigma^{-1}$ to the coefficients of the minimal polynomial {of $x^{1/\ell}$ over $E$}, we see that $x^{1/q\ell}$ is also algebraic over $ E$ of degree $N$. Since $E(x^{1/\ell}) \subset E(x^{1/q\ell})$ we have $E(x^{1/\ell}) = E(x^{1/q\ell})$. In this way we see that $E(x^{1/\ell}) = E(x^{1/q^m\ell})$ for all $m \in \mathbb{N}$.  Since $E$ is finitely generated over $C$, we have that $E(x^{1/\ell})$ is also finitely generated over $C$. This implies that $C(\{ x^{1/q^m\ell} \ | \ m \in \mathbb{N}\}) \subset E(x^{1/\ell})$ is finitely generated over $C$, and in particular {that} the degree of $C(\{ x^{1/q^m\ell} \ | \ m \in \mathbb{N}\})$ over $C(x^{1/\ell})$ {is} bounded. This contradiction implies that $E = L \cap k = F$, and so $G=H$. \end{proof}

In the next section we apply Proposition~\ref{sigmagp} to completely characterize those linear algebraic groups that occur as Galois groups for \emph{integrable} linear difference equations.

\section{Integrable systems} \label{int-sec}

In the {\em differential} Galois theory of linear difference equations, we let $(k, \sigma, \d)$ be a $\SD$-field, that is, a field $k$ with an automorphism $\sigma$ and a derivation $\d$ such that $\sigma \d = \d \sigma$.  Following the presentation in \cite{HaSi08}, we assume that $C = k^\sigma $ is a differentially closed field\footnote{See \cite{HaSi08} for definitions and properties. Presentations assuming only that $C$ is algebraically closed are given in \cite{DiHa12, Wib12}.}. For a difference equation $\sigma(Y) = AY$ with $A\in \GL_n(k)$, one has a {\it $\SD$-Picard-Vessiot ring} $R=k\{Z,\frac{1}{\det Z}\}_\d$, which is a $\SD$-simple $k$-algebra differentially generated by the entries of a fundamental solution matrix $Z$, the inverse of its determinant, and all their derivatives. This ring is unique up to $\SD$-$k$-isomorphism and has no nilpotent elements but may have zero divisors. Its total ring of quotients $K$ is called the {\it total $\SD$-Picard-Vessiot ring} and has the property that $K^\sigma = k^\sigma$. The group $G$ of $\SD$-$k$-algebra automorphisms is called the {\it $\SD$-Galois group} and can be identified with a linear {\it differential} algebraic group over $C$, that is, $G\subset \GL_n(C)$ is defined by a set of algebraic differential equations over $C$. Finally, there is a Galois correspondence between differential subgroups of the $\SD$-Galois group and $\SD$-subrings of $K$, containing $k$, all of whose nonzero divisors are invertible.

 Note that any linear algebraic group $G$ defined over $C$ is {\it a fortiori} a linear differential algebraic group. Given a linear differential algebraic group $G$, one can consider the set of elements of $G$ whose entries belong to $C_0 = C^\d = \{c \in C \ | \ \d(c) = 0\}$, and this set is also a linear differential algebraic group. We will distinguish these two groups by denoting the first group by $G(C)$ and the second group by $G(C_0)$.

The following result, proved in \cite[Prop.~6.21]{HaSi08}, compares the $\sigma$-Galois group and the $\SD$-Galois group of an equation $\sigma(Y) = AY,$ $A\in \GL_n(k)$ over $k$.

\begin{prop}\label{propdense}\hspace{-.12in} \begin{enumerate} \item If $R = k\{Z,\frac{1}{\det Z}\}_\d$ is a $\SD$-Picard-Vessiot ring for  $\sigma(Y) = AY,$ $A\in \GL_n(k)$, then $S= k[Z,\frac{1}{\det Z}]$ is a $\sigma$-Picard-Vessiot ring for this equation.

\item With the above identification, the $\SD$-Galois group of $\sigma(Y) = AY$ is a Zariski-dense subgroup of the $\sigma$-Galois group of this equation. \end{enumerate}
\end{prop}

\begin{defn}\label{intdef}
Given $A,\tilde{A}\in\GL_n(k)$, we say that the systems $\sigma(Y)=AY$ and $\sigma(Y)=\tilde{A}Y$ are {\em equivalent} if there exists $T\in\GL_n(k)$ such that $\tilde{A}=\sigma(T)AT^{-1}$.

In this case, if $Z$ is a fundamental solution matrix for the first system then $\tilde{Z}:=TZ$ is a fundamental solution matrix for the second system, and therefore they have the same Picard-Vessiot rings and Galois groups.\end{defn}

\begin{defn} The system $\sigma(Y)=AY$ with $A\in\mathrm{GL}_n(k)$ is \emph{integrable} over $k$ if one of the following equivalent conditions is satisfied:
\begin{enumerate}
\item The $\sigma\delta$-Galois group $G$ for $\sigma(Y)=AY$ over $k$ is $\delta$-constant, i.e., $G$ is conjugate to a subgroup of $\GL_n(C_0)$.
\item There exists $B\in\mathfrak{gl}_n(k)$ such that \begin{equation} \label{iso} \sigma(B)=ABA^{-1}+\delta(A)A^{-1}.\end{equation} \end{enumerate}\end{defn}

The equivalence of these two conditions is proved in \cite[Prop.~2.9]{HaSi08}\footnote{There is an error in the proof of this proposition that is fixed in the Erratum of \cite{HaSi08}.}. Note that if $\sigma(Y) = AY$ is integrable then after a change of fundamental solution matrix we may assume that its $\sd$-Galois group $G$ is in $\GL_n(C_0)$. Since the derivation is trivial on $C_0$, we have that $G$ is actually a linear algebraic group defined over $C_0$.  Its Zariski closure in $GL_n(C)$ must therefore be $G(C)$.  This and Proposition~\ref{propdense} imply the following result.

\begin{prop}\label{intprop} If the linear difference equation $\sigma(Y) = AY$ with $A\in \GL_n(k)$ is integrable, then its $\sigma$-Galois group is $G(C) \subset \GL_n(C)$ if and only if its $\SD$-Galois group is $G(C_0)$.\end{prop}

In this section we will completely characterize the Galois groups of integrable linear difference equations in the following three cases of triples $(k,\sigma,\delta)$ as above, corresponding to the three cases of $\sigma$-fields discussed in the previous section. However, in this section we will always assume that $C$ is a differentially closed field of characteristic zero, instead of just algebraically closed. It is known that $C_0=C^\delta$ is algebraically closed \cite[\S9.1]{CaSi}.\vspace{.1in}

\begin{itemize}
\item[{\rm case S:}] The field $k:=C(x)$, $\sigma$ is the shift operator defined by $\sigma(x)=x+1$, and the derivation is defined by $\delta(x)=x$.
\item[{\rm case Q:}] The field $k:=C(x)$, $\sigma$ is the $q$-dilation operator defined by $\sigma(x)=qx$ for some $q\in C$, $q\neq0$ and not a root of unity, and the derivation is defined by $\delta(x)=x$.
\item[{\rm case M:}] The field $k:=C(\mathrm{log}(x),\{x^{1/\ell}\ |\ \ell\in\mathbb{N}\})$, $\sigma$ is the Mahler operator defined by $\sigma(x^\alpha)=x^{q\alpha}$ and $\sigma(\mathrm{log}(x))=q\mathrm{log}(x)$ for some integer $q\geq2$, and the derivation is defined by $\delta(x^\alpha)=\alpha\,  \mathrm{log}(x)\,x^\alpha$ and $\delta(\mathrm{log}(x))=\mathrm{log}(x)$.
\end{itemize}

\vspace{.1in}

Our point of departure is the following result, which is contained in Theorem 1 and Corollary 7 of \cite{SS16}\footnote{N.B.: our choice of notation for $A$ and $B$ here is the opposite from that of\cite{SS16}.}:

\begin{thm}(Sch\"afke-Singer, \cite{SS16}) \label{ss-thm-verbatim} Let $A\in\mathrm{GL}_n(C_0(x))$ and $B\in\mathfrak{gl}_n(C_0(x))$, and suppose that the system \begin{equation} \label{ss-eq}  \begin{cases} \sigma(Y)=AY \\ \partial(Y)=BY\end{cases}\end{equation} satisfies the consistency condition \begin{equation} \label{ss-int} \mu\sigma(B)A=AB+\partial(A),\end{equation} where $\partial$ and $\mu$ are defined as follows: in case S, $\mu=1$ and $\partial=\delta$; in case Q, $\mu=1$ and $\partial=\delta$; and in case M, $\mu=q$ and $\partial=x \ \frac{d}{dx}=(\mathrm{log}(x))^{-1}\delta$. Then, in each of these cases, the system $\sigma(Y)=AY$ is equivalent over $k$ to a system $\sigma(Y)=\tilde{A}Y$, where $\tilde{A}\in\mathrm{GL}_n(C_0)$.
\end{thm}

In the following result we show that the hypotheses of Theorem~\ref{ss-thm-verbatim} can be relaxed somewhat, and related to the integrability of the difference system.

\begin{prop} \label{SS-thm} Suppose we are in case S, case Q, or case M. If the system $\sigma(Y)=AY$ with $A\in\mathrm{GL}_n(C_0(x))$ is integrable over $k$, then it is equivalent over $k$ to a system $\sigma(Y)=\tilde{A}Y$ with $\tilde{A}\in\GL_n(C_0)$. \end{prop}

\begin{proof} {We begin by observing that \eqref{iso} and \eqref{ss-int} are equivalent in cases S and Q, whereas in case M we have that $B\in\gl_n(k)$ satisfies \eqref{ss-int} if and only if $\mathrm{log}(x)B$ satisfies \eqref{iso}.} Let us first consider cases S and Q. To see that the existence of $B\in\gl_n(C(x))$ satisfying \eqref{ss-int} implies the existence of $B'\in \gl_n(C_0(x))$ satisfying the same equation, observe that \eqref{ss-int} defines a system of algebraic equations for the unknown coefficients of the matrix entries of $B$, considered as rational functions in $x$. The algebraic variety $V$ defined by this system of equations is defined over $C_0$, because $A\in \GL_n(C_0(x))$. Since $C_0$ is algebraically closed, if $V$ has a $C$-point, then it must also have a $C_0$-point, which yields $B'\in\gl_n(C_0(x))$ satisfying \eqref{ss-int}. {This concludes the proof in cases S and Q, by Theorem~\ref{ss-thm-verbatim}.}

In case M, we still have to show that the existence of $B\in\gl_n(k)$ satisfying \eqref{iso} implies the existence of $B'\in\gl_n(C(x))$ satisfying \eqref{ss-int}. To see this, first note that there is an integer $\ell\in\mathbb{N}$ such that $B\in\gl_n(C(x^{1/\ell},\mathrm{log}(x))$. Since $\mathrm{log}(x)$ is transcendental over $C(x^{1/\ell})$, we may write $B$ as a Laurent series $\sum_{i\geq N} B_i \mathrm{log}^i(x)$, where $B_i\in\mathfrak{gl}_n(C(x^{1/\ell}))$ and $\mathrm{log}^i(x):=(\mathrm{log}(x))^i$. Substituting this expression for $B$ in \eqref{iso}, we obtain \begin{align*} \sigma(\sum B_i\mathrm{log}^i(x))A &= A\sum B_i\mathrm{log}^i(x) + \delta(A) \\
\sum q^i\mathrm{log}^i(x)\sigma(B_i)A &= \sum \mathrm{log}^i(x)AB_i + \mathrm{log}(x)\partial(A), \intertext{whence the equality of $\mathrm{log}^i(x)$ terms for $i=1$:}
q\mathrm{log}(x)\sigma(B_1)A &= \mathrm{log}(x)(AB_1 +\partial(A)),
\end{align*}
which implies that $B_1\in\gl_n(C(x^{1/\ell}))$ satisfies \eqref{ss-int}.

{Suppose that $\ell\in\mathbb{N}$ is the smallest integer such that $B_1\in\gl_n(C(x^{1/\ell}))$, and let $d:=\ell/\mathrm{gcd}(q,\ell)$. Since $A\in\GL_n(C_0(x))$ and $\sigma(B_1)\in\gl_n(C(x^{1/d}))$, it follows that \[ B_1=qA^{-1}\sigma(B_1)A-A^{-1}\partial(A)\in\gl_n(C(x^{1/d})).\] Since $\ell$ is minimal and $d\leq\ell$, it follows that $d=\ell$, that is, $\ell$ and $q$ are relatively prime. Now let $\tau$ be a cyclic generator of  $\mathrm{Gal}(C(x^{1/\ell})/C(x))$, so that $\tau(x^{1/\ell})=\zeta_\ell x^{1/\ell}$, where $\zeta_\ell$ is a primitive root of unity. Then we see that $\tau\sigma=\sigma\tau^q$, and since the $q$-power map is an automorphism of the group $\mathrm{Gal}(C(x^{1/\ell})/C(x))\simeq \mu_\ell$ of $\ell$-th roots of unity, we see that the trace operator $\mathrm{Tr}:=\sum_{i=0}^{\ell-1}\tau^i:C(x^{1/\ell})\rightarrow C(x)$ commutes with $\sigma$. Applying $\mathrm{Tr}$ to both sides of \[q\sigma(B_1)A=AB_1+\partial(A),\] we find that $\tilde{B}_1:=\mathrm{Tr}(B_1)/\ell\in\gl_n(C(x))$ satisfies \eqref{ss-int}.} Arguing as in cases S and Q {we} conclude that there exists $B'\in\gl_n(C_0(x))$ satisfying \eqref{ss-int} in case M as well. \end{proof}

The main result of this section is the following.

\begin{thm} \label{thm4}  Consider a difference system of the form \begin{equation}\label{eq}\sigma(Y)=AY, \ \ \text{with} \ A\in\GL_n(C_0(x)).\end{equation} \begin{enumerate}

\item In case S, a {$\delta$-}constant group $G \subset \mathrm{GL}_n(C_0)$ is the $\sigma\delta$-Galois group of \eqref{eq} over $k$ if and only if it is isomorphic to a group of the form 
\[\mathbb{G}_m(C_0)^r  \times \mathbb{Z}/t \mathbb{Z} \]
where $r$ is a nonnegative integer and $t$ is a positive integer.

\item In case Q, a {$\delta$-}constant group $G \subset \mathrm{GL}_n(C_0)$ is the $\sigma\delta$-Galois group of \eqref{eq} over $k$ if and only if it is isomorphic to a group of the form 
\[\mathbb{G}_m(C_0)^r \times \mathbb{G}_a(C_0)^s \times \mathbb{Z}/t\mathbb{Z} \times \mathbb{Z}/p\mathbb{Z}\]
where $r$ and $s$ are nonnegative integers, $s \leq 1$, and $t$ and $p$ are positive integers.

\item In case M, a {$\delta$-}constant group $G \subset \mathrm{GL}_n(C_0)$ is the  $\sigma\delta$-Galois group  of \eqref{eq} over $k$ if and only if it is isomorphic to a group of the form 
\[\mathbb{G}_m(C_0)^r \times \mathbb{G}_a(C_0)^s \times \mathbb{Z}/t\mathbb{Z} \times \mathbb{Z}/p\mathbb{Z}\]
where $r$ and $s$ are nonnegative integers, $s \leq 1$, and $t$ and $p$ are positive integers.\end{enumerate}
\end{thm}

\begin{proof} {(2)}~Assume that a {$\delta$-}constant group $G \subset \mathrm{GL}_n(C_0)$ is the $\sigma\delta$-Galois group of \eqref{eq}. Then $\sigma(Y) = AY$ is integrable with respect to the derivation $\delta = x\frac{d}{dx}$. Proposition~\ref{SS-thm} impiles that $\sigma(Y) = AY$ is equivalent to a system $\sigma(Y) = \tilde{A} Y$, with $\tilde{A} \in \mathrm{GL}_n(C_0) \subset \mathrm{GL}_n(C)$. Proposition~\ref{sigmagp} implies that the $\sigma$-Galois group of $\sigma(Y) = \tilde{A} Y$ is isomorphic to a group of the form $\mathbb{G}_m(C_0)^r \times \mathbb{G}_a(C_0)^s \times \mathbb{Z}/t\mathbb{Z} \times \mathbb{Z}/p\mathbb{Z}$, with $r,s,t,p$ as above.  Therefore Proposition~\ref{intprop} yields the conclusion.

Now assume that $G$ is isomorphic to a group of the described form. Proposition~\ref{sigmagp} implies that there is $A \in \mathrm{GL}_n(C_0)$ such {the $\sigma$-Galois group of $\sigma(Y) = AY $ over $C_0(x)$ is} $\mathbb{G}_m(C_0)^r \times \mathbb{G}_a(C_0)^s \times \mathbb{Z}/t\mathbb{Z} \times \mathbb{Z}/p\mathbb{Z}$. {By \cite[Prop.~2.4 and Cor.~2.5]{CHS},} this equation has $\mathbb{G}_m(C)^r \times \mathbb{G}_a(C)^s \times \mathbb{Z}/t\mathbb{Z} \times \mathbb{Z}/p\mathbb{Z}$ as its $\sigma$-Galois group over $k$. Since $\delta(A) = 0$, this system is integrable. {By} Proposition~\ref{intprop},  its $\sigma\delta$-Galois group is $\mathbb{G}_m(C_0)^r \times \mathbb{G}_a(C_0)^s \times \mathbb{Z}/t\mathbb{Z} \times \mathbb{Z}/p\mathbb{Z}$.

{(1)~and~(3)} The proofs of these statements are the same, {\it mutatis mutandis}, as that of {(2)}.\end{proof}

\section{Projectively integrable systems} \label{proj-sec} 
In this section we extend some of the conclusions of the previous section for integrable systems to the larger class of {\em projectively integrable} systems. We begin by proving some preliminary results for an arbitrary $\sigma\delta$-field $k$ with $C:=k^\sigma$ differentially closed and $C_0:=C^\delta$, before specializing to the cases S, Q, and M defined in the previous section. We write $\mathrm{Scal}_n(C)\subset\GL_n(C)$ for the group of invertible scalar matrices, and $I_n$ for the $n\times n$ identity matrix.

\begin{defn}\label{projintdef}
The system $\sigma(Y)=AY$ with $A\in\mathrm\GL_n(k)$ is \emph{projectively integrable} {over $k$} if one of the following equivalent conditions is satisfied:
\begin{enumerate}
\item The $\sigma\delta$-Galois group $G$ for $\sigma(Y)=AY$ over $k$ is projectively {$\delta$-}constant, i.e., {$G$ is} conjugate to a subgroup of $\mathrm{GL}_n(C_0)\cdot\mathrm{Scal}_n(C)=\mathrm{SL}_n(C_0)\cdot\mathrm{Scal}_n(C)$.
\item There exists $B\in\mathfrak{gl}_n(k)$ such that \begin{equation}\label{piso}\sigma(B)=ABA^{-1}+\delta(A)A^{-1}-\tfrac{1}{n}\delta(\mathrm{det}(A))\mathrm{det}(A)^{-1}\cdot I_n.\end{equation}
\end{enumerate} \end{defn}

The equivalence of these two conditions follows, {\it mutatis mutandis}, from the proof of {\cite[Prop.~2.10]{DHR15}}.

We now  show how certain  questions concerning projective integrability can be reduced to questions concerning integrability. To do this we show that, given a projectively integrable system, there is a {naturally} associated  system that is integrable. The following result makes precise the statement made in the introduction that a projectively integrable system becomes integrable after ``moding out by scalars."

\begin{prop} \label{piso-seq-prop} Suppose that $\sigma(Y)=AY$ with $A\in\GL_n(k)$ is projectively integrable, and let $H$ and $G$ denote its $\sigma$-Galois group and $\sigma\delta$-Galois group over $k$, respectively. Let 
\[\tilde{A}:= \mathrm{det}(A)^{-1}\otimes(A^{\otimes n}).\] Then:
\begin{enumerate}
\item The system ${\sigma}(Y) = \tilde{A}Y$ is integrable.
\item If $\tilde{H}$ and $\tilde{G}$ are the $\sigma$-Galois group and $\sigma\delta$-Galois group, respectively, of ${\sigma}(Y) = \tilde{A}Y$, we have short exact sequences
\begin{equation}\label{piso-seq}1\rightarrow \mathrm{Scal}_n(C)\cap H\rightarrow H \rightarrow \tilde{H}\rightarrow 1\qquad\text{and}\qquad 1\rightarrow \mathrm{Scal}_n(C)\cap G\rightarrow G \rightarrow \tilde{G}\rightarrow 1.\end{equation}
\end{enumerate}
 \end{prop}
 
 \begin{proof} (2)~Let $V$ denote the ${\sigma}$-module associated {with} $\sigma(Y)=AY$. Since the difference module associated {with} $\sigma(Y)=\tilde{A}Y$ is given by \[\tilde{V}:=(\bigwedge^nV)^\vee\otimes_kV^{\otimes n},\] the Tannakian formalism implies that the homomorphism $\rho:H\twoheadrightarrow\tilde{H}\subset\mathrm{GL}_{n^n}(C)$ corresponding to the action of $H$ on the solution space for $\tilde{V}$ is given by \[\rho(T)= \mathrm{det}(T)^{-1}\otimes (T^{\otimes n}).\] Note that $\rho(T)=I_{n^n}$ if and only if $T$ is a scalar matrix.

(1)~We remark that $G$ is projectively {$\delta$-}constant if and only if \begin{equation}\label{proj-const-rel}\delta(T)=\frac{\delta(\mathrm{det}(T))}{n\cdot\mathrm{det}(T)}\cdot T\end{equation} for every $T\in G$. In the following computation we adopt the convention that $T^{\otimes 0} =(1)$. Let us assume for the moment that $G\subseteq\mathrm{GL}_n(C_0)\cdot\mathrm{Scal}_n(C)$. {It follows from \eqref{proj-const-rel} that} \[\delta(T^{\otimes n})=\sum_{i=0}^{n-1} T^{\otimes i}\otimes\delta(T)\otimes T^{\otimes (n-i-1)}=\sum_{i=0}^{n-1}T^{\otimes i}\otimes(\tfrac{\delta(\mathrm{det}(T))}{n\cdot\mathrm{det}(T)}\cdot T)\otimes T^{\otimes (n-i-1)}=\tfrac{\delta(\mathrm{det}(T))}{\mathrm{det}(T)}\cdot T^{\otimes n}.\] {for each $T\in G$}. Therefore \begin{align*}\delta(\mathrm{det}(T)^{-1}\otimes T^{\otimes n}) &=(\delta(\mathrm{det}(T)^{-1}))\otimes T^{\otimes n} + \mathrm{det}(T)^{-1}\otimes\delta (T^{\otimes n}) \\ &=(-\tfrac{\delta(\mathrm{det}(T))}{\mathrm{det}(T)}\cdot\mathrm{det}(T)^{-1})\otimes T^{\otimes n}+\mathrm{det}(T)^{-1}\otimes (\tfrac{\delta(\mathrm{det}(T))}{\mathrm{det}(T)}\cdot T^{\otimes n}) \\ &=(-\tfrac{\delta(\mathrm{det}(T))}{\mathrm{det}(T)}+\tfrac{\delta(\mathrm{det}(T))}{\mathrm{det}(T)})\cdot(\mathrm{det}(T)^{-1}\otimes T^{\otimes n})=0,\end{align*} {so} $\tilde{G}=\rho(G)\subset\mathrm{GL}_{n^n}(C_0)$ {provided} that $G\subseteq\mathrm{GL}_n(C_0)\cdot\mathrm{Scal}_n(C)$. {In general, if $D\in\mathrm{GL}_n(C)$ is such that} $DGD^{-1}\subseteq\mathrm{GL}_n(C_0)\cdot\mathrm{Scal}_n(C)$ then we have just shown that $\rho(D)\tilde{G}\rho(D)^{-1}\subset \mathrm{GL}_{n^n}(C_0)$, and therefore $\tilde{G}$ is $\delta$-constant, as we wanted to show.\end{proof}

\begin{prop} \label{iso-piso} Let $G\subseteq\mathrm{GL}_n(C)$ denote the $\sigma\delta$-Galois group for the system \begin{equation}\label{iso-piso-sys}\sigma(Y)=AY, \quad \text{with} \ \ A\in\GL_n(k).\end{equation} The system is integrable if and only if it is projectively integrable and any of the following conditions is satisfied:
\begin{enumerate}
\item The system $\sigma(y)=\mathrm{det}(A)y$ is integrable.
\item $\det(G)\subseteq\mathbb{G}_m(C_0)$.
\item There exists $b\in k$ such that $\sigma(b)-b=\delta(\det(A))\mathrm{det}(A)^{-1}$.
\item $G\cap\mathrm{Scal}_n(C)\subseteq\mathrm{Scal}_n(C_0)$.
\end{enumerate} \end{prop}

\begin{proof}
Since $\mathrm{det}(G)$ is the $\sigma\delta$-Galois group over $k$ for the system $\sigma(y)=\mathrm{det}(A)y$, it is clear that (1) $\Leftrightarrow$ (2) $\Leftrightarrow$ (3). Moreover, since the restriction of $\mathrm{det}$ to $G\cap\mathrm{Scal}_n(C)$ coincides with the $n$-power map on $\mathrm{Scal}_n(C)\simeq\mathbb{G}_m(C)$, we have that (2) $\Rightarrow$ (4).

If \eqref{iso-piso-sys} is integrable, then $G$ is conjugate to a subgroup of $\GL_n(C_0)\subset\GL_n(C_0)\cdot\mathrm{Scal}_n(C)$, so the system is also projectively integrable and $\mathrm{det}(G)\subseteq\mathbb{G}_m(C_0)$. Hence, the integrability of \eqref{iso-piso-sys} implies that condition (2) is satisfied, and therefore conditions (1), (3), and (4) are satisfied as well.

At this point it would be sufficent to show that if \eqref{iso-piso-sys} is projectively integrable and (4) holds then \eqref{iso-piso-sys} is integrable. However, let us first give a simple proof that if \eqref{iso-piso-sys} is projectively integrable and (3) holds then \eqref{iso-piso-sys} is integrable. So suppose that \eqref{iso-piso-sys} is projectively integrable and $b\in k$ satisfies $\sigma(b)-b=\delta(\mathrm{det}(A))\mathrm{det}(A)^{-1}.$ Suppose $B\in\mathfrak{gl}_n(k)$ satisfies \eqref{piso}, and let $B':=B+\frac{b}{n}\cdot I_n$. Then \begin{align*} \sigma(B') &= \sigma(B) + \tfrac{\sigma(b)}{n}\cdot I_n \\ &=ABA^{-1} +\delta(A)A^{-1} -\tfrac{1}{n}\delta(\mathrm{det}(A))\mathrm{det}(A)^{-1}\cdot I_n +\tfrac{\sigma(b)}{n}\cdot I_n \\ &=A(B+\tfrac{b}{n}\cdot I_n)A^{-1} + \delta(A)A^{-1}-\tfrac{1}{n}\delta(\mathrm{det}(A))\mathrm{det}(A)^{-1}\cdot I_n +\tfrac{\sigma(b)-b}{n}\cdot I_n \\ &= AB'A^{-1}+\delta(A)A^{-1}, \end{align*}
and therefore \eqref{iso-piso-sys} is integrable.

Let us now show that if \eqref{iso-piso-sys} is projectively integrable and (4) holds then \eqref{iso-piso-sys} is also integrable. Let us assume without loss of generality that $G\subset\mathrm{GL}_n(C_0)\cdot\mathrm{Scal}_n(C)$. Let $H$ and $\tilde{H}$ denote the $\sigma$-Galois groups for $\sigma(Y)=AY$ and $\sigma(Y)=\tilde{A}Y$, respectively, where $\tilde{A}:=\mathrm{det}(A)^{-1}\otimes(A^{\otimes n})$. Observe that $H\subset\mathrm{GL}_n(C)$ and $\tilde{H}\subset\mathrm{GL}_{n^n}(C)$ are both algebraic groups defined over $C_0$. By Proposition~\ref{piso-seq-prop}, the system $\sigma(Y)=\tilde{A}Y$ is integrable and its $\sigma\delta$-Galois group $\tilde{G}\subset\mathrm{GL}_{n^n}(C_0)$. By Proposition~\ref{intprop}, $\tilde{G}$ can be identified with the $C_0$-points of $\tilde{H}$, and we may write $\tilde{G}=\tilde{H}(C_0)$.

Let $R$ be the $\sigma\delta$-Picard-Vessiot ring for \eqref{iso-piso-sys} over $k$, and let $S\subset R$ be the $\sigma$-Picard-Vessiot ring, both generated by the same choice of fundamental solution matrix $Z$.

We will first make the supplementary assumption that $R$ is a domain. In this case, $L:=\mathrm{Frac}(S)$ and $K:=\mathrm{Frac}(R)$ are fields. If we let \[\tilde{R}:=k\bigl\{\mathrm{det}(Z)^{-1}\otimes (Z^{\otimes n})\bigr\}_\delta\] then $\tilde{R}\subset R$ is a $\sigma\delta$-Picard-Vessiot ring for $\sigma(Y)=\tilde{A}Y$ over $k$. Since this equation is integrable, $\tilde{R}$ is also a $\sigma$-Picard-Vessiot ring for the same equation over $k$. Moreover, since $\tilde{R}\subset S\subset R$, $\tilde{R}$ is also a domain. Let $\tilde{k}:=\mathrm{Frac}(\tilde{R})$. Then $\tilde{k}$ is the fixed field of $H\cap\mathrm{Scal}_n(C)$ in $L$, and therefore $H\cap\mathrm{Scal}_n(C)$ is the $\sigma$-Galois group for \eqref{iso-piso-sys} over $\tilde{k}$. Similarly, $\tilde{k}$ is the fixed field for $G\cap\mathrm{Scal}_n(C)$ in $K$, and therefore $G\cap\mathrm{Scal}_n(C)$ is the $\sigma\delta$-Galois group for \eqref{iso-piso-sys} over $\tilde{k}$. It follows from Proposition~\ref{propdense} that $G\cap\mathrm{Scal}_n(C)$ is Zariski-dense in $H\cap\mathrm{Scal}_n(C)$. If $G\cap\mathrm{Scal}_n(C)\subseteq\mathrm{Scal}_n(C_0)$ then $G\cap\mathrm{Scal}_n(C)$ coincides with the $C_0$ points of $H\cap\mathrm{Scal}_n(C)$, which latter group is defined over $C_0$. Let us write $D:=H\cap\mathrm{Scal}_n(C)$, so that $G\cap\mathrm{Scal}_n(C_0)=D(C_0)$. From the short exact sequence \[ 1\rightarrow H\cap\mathrm{Scal}_n(C) \rightarrow H\rightarrow \tilde{H}\rightarrow 1\] of Proposition~\ref{piso-seq-prop}(2), we obtain the short exact sequence of $C_0$-points \begin{equation}\label{tilde-h-seq}1\rightarrow D(C_0)\rightarrow H(C_0)\rightarrow \tilde{H}(C_0)\rightarrow 1.\end{equation} On the other hand, we have shown that the other short exact sequence \[ 1\rightarrow G\cap\mathrm{Scal}_n(C_0)\rightarrow G \rightarrow \tilde{G}\rightarrow 1\] of Proposition~\ref{piso-seq-prop}(2) may be rewritten as \begin{equation}\label{tilde-g-seq}1\rightarrow D(C_0) \rightarrow G \rightarrow \tilde{H}(C_0)\rightarrow 1.\end{equation} It follows from \eqref{tilde-h-seq} and \eqref{tilde-g-seq} that the images of $H(C_0)\subset\mathrm{GL}_n(C)$ and $G\subset\mathrm{GL}_n(C)$ in $\mathrm{GL}_n(C)/D(C_0)$ are equal, and therefore $G=\tilde{H}(C_0)$. In particular, $G$ is $\delta$-constant.

In general, when $R$ is not assumed to be a domain, we proceed as follows. By \cite[Cor.~1.16]{PuSi} and \cite[Lem.~6.8]{HaSi08}, there exists a positive integer $t$ such that $R=R_0\oplus\dots\oplus R_{t-1}$ (resp., $S=S_0\oplus\dots\oplus S_{t-1}$), where each $R_i$ (resp., $S_i$) is a domain, and a $\sigma^t\delta$-Picard-Vessiot extension (resp., $\sigma^t$-Picard-Vessiot extension) for $\sigma^t(Y)=A_tY$ over $k$ considered as a $\sigma^t\delta$-field, where $A_t:=\sigma^{t-1}(A)\dots\sigma(A)A$. If we let $G_t$ and $H_t$ denote the $\sigma^t\delta$-Galois group and $\sigma^t$-Galois group, respectively, for $\sigma^t(Y)=A_tY$ over $k$, still considered as a $\sigma^t\delta$-field, we have short exact sequences \begin{equation} \label{ses}1\rightarrow H_t\rightarrow H\rightarrow\mathbb{Z}/t\mathbb{Z}\rightarrow 1 \qquad\text{and} \qquad1\rightarrow G_t\rightarrow G\rightarrow\mathbb{Z}/t\mathbb{Z}\rightarrow 1.\end{equation} Since $\sigma^t(Y)=A_tY$ is also projectively integrable and $G_t\cap\mathrm{Scal}_n(C)\subseteq\mathrm{Scal}_n(C_0)$, the argument in the previous paragraph shows that $G_t$ coincides with the $C_0$-points of $H_t$. It follows from \eqref{ses} that $G$ also coincides with the $C_0$-points of $H$, whence $G$ is $\delta$-constant in this case also.
\end{proof}

\begin{cor} \label{nonint-scal-cor}Suppose $\sigma(Y)=AY$ with $A\in\GL_n(k)$ is projectively integrable but not integrable, and let $H$ and $G$ denote its $\sigma$-Galois group and $\sigma\delta$-Galois group over $k$, respectively. Then $H\cap\mathrm{Scal}_n(C)=\mathrm{Scal}_n(C)$ and $G\cap\mathrm{Scal}_n(C)\supsetneqq\mathrm{Scal}_n(C_0)$.\end{cor}

\begin{proof} If $G\cap\mathrm{Scal}_n(C)$ were finite then this group would be {$\delta$-}constant, and so the equation would be integrable by Proposition~\ref{iso-piso}. Therefore we may assume that $G \cap \mathrm{Scal}_n(C)$ and $H \cap \mathrm{Scal}_n(C)$ are infinite, {in which case $H\cap\mathrm{Scal}_n(C)=\mathrm{Scal}_n(C)$.} It follows from \cite[Prop.~31]{Cassidy72} that any infinite $\delta$-subgroup of $\mathbb{G}_m(C)\simeq\mathrm{Scal}_n(C)$ contains $\mathbb{G}_m(C_0)$. By Proposition~\ref{iso-piso}, the containment $\mathrm{Scal}_n(C_0)\subset G\cap\mathrm{Scal}_n(C)$ must be proper. \end{proof}

\begin{lem} \label{det} Assume that we are in case S, Q, or M. Suppose that $\sigma(Y)=AY$ with $A\in \GL_n(C_0(x))$ is projectively integrable but not integrable, and let $G$ be its $\sigma\delta$-Galois group. Then $G\cap\mathrm{Scal}_n(C)$ and $\mathrm{det}(G)$ coincide as subgroups of $\mathbb{G}_m(C)$.

In particular, in cases S and M we have that $G\cap\mathrm{Scal}_n(C)=\mathbb{G}_m(C)$, whereas in case Q we have that $G\cap\mathrm{Scal}_n(C)$ {is either} $\mathbb{G}_m(C)$ or $\{c\in C^\times \ | \ \delta(\frac{\delta(a)}{a})=0\}$.\end{lem}

\begin{proof} Since $\sigma(Y)=AY$ is not integrable, it follows from Proposition~\ref{iso-piso} that $\mathrm{det}(G)\not\subset\mathbb{G}_m(C_0)$. In cases S and M this forces $\mathrm{det}(G)=\mathbb{G}_m(C)$, by \cite[Cor.~3.4(1)]{HaSi08} in case S and \cite[Prop.~3.1]{DHR15} in case M (since $\mathrm{det}(A)\in C_0(x)$). In case Q we know that either $\mathrm{det}(G)=\{a\in C^\times \ | \ \delta(\tfrac{\delta(a)}{a})=0\}$ or else $\mathrm{det}(G)=\mathbb{G}_m(C)$, by \cite[Cor.~3.4(2) and Prop.~4.3(2)]{HaSi08}.

First suppose that $\mathrm{det}(G)=\mathbb{G}_m(C)$. Then the $\delta$-type of $G$ is at least $1$. Since $\tilde{G}$ is $\delta$-constant, its $\delta$-type is $0$. It now follows that $G\cap\mathrm{Scal}_n(C)$ has $\delta$-type $1$, for otherwise it would follow from Proposition~\ref{piso-seq-prop} that $G$ has $\delta$-type $0$, a contradiction. This shows that $G\cap\mathrm{Scal}_n(C)=\mathrm{Scal}_n(C)$, because any proper $\delta$-subgroup of $\mathbb{G}_m(C)$ has $\delta$-type $0$.

Now suppose that $\mathrm{det}(G)=\{a\in C^\times \ | \ \delta(\tfrac{\delta(a)}{a})=0\}=:W$. Since $\mathrm{det}|_{\mathrm{Scal}_n(C)}$ coincides with the $n$-power map and $G\cap\mathrm{Scal}_n(C)$ is infinite, and therefore divisible, it follows that $G\cap\mathrm{Scal}_n(C)\subseteq W$. If $L$ is a linear differential operator then any $\delta$-subgroup of the group $\{a \in C \ | \ L(a) = 0\}$ is of the form $\{a \in C \ | \ L_1(a) = 0\}$ where $L_1$ is a right factor of $L$.  Therefore the infinite $\delta$-subgroups of $W$ are defined by right-hand factors of the operator $\delta$, but these are only $\delta$ and $1$. It follows that either $G\cap\mathrm{Scal}_n(C)=W$, or else $G\cap\mathrm{Scal}_n(C)=\mathrm{Scal}_n(C_0)$. It now follows from Corollary~\ref{nonint-scal-cor} that $G\cap\mathrm{Scal}_n(C)=W$. \end{proof}

The following technical lemma will allow us to restrict the possible groups that may occur as Galois groups for projectively integrable linear difference equations in cases S, Q, and M.

\begin{lem} \label{tech-lem} Suppose that $H$ is a linear algebraic group such that $\mathrm{Scal}_n(C)\subset H \subset \mathrm{GL}_n(C)$ and \[H/\mathrm{Scal}_n(C)\simeq \mathbb{G}_m(C)^r\times\mathbb{G}_a(C)^s\times\mathbb{Z}/t\mathbb{Z}\times\mathbb{Z}/p\mathbb{Z},\] where $r$ and $s$ are nonnegative integers, $s\leq 1$, and $t$ and $p$ are positive integers.

Then the connected component of the identity $H^0\simeq\mathbb{G}_m(C)^{r+1}\times\mathbb{G}_a(C)^s$, and there are exact sequences \begin{gather} \label{tech-ses-1} 1\longrightarrow H^0\times\mathbb{Z}/t\mathbb{Z}\longrightarrow H \longrightarrow \mathbb{Z}/p\mathbb{Z}\longrightarrow 1 \intertext{and} \label{tech-ses-2}
1\longrightarrow H^0\times\mathbb{Z}/p\mathbb{Z}\longrightarrow H \longrightarrow \mathbb{Z}/t\mathbb{Z}\longrightarrow 1. \end{gather} Moreover, if $\mathrm{gcd}(n,t,p)=1$, then both exact sequences are split, and \begin{equation}\label{tech-ses-3}H\simeq\mathbb{G}_m(C)^{r+1}\times\mathbb{G}_a(C)^s\times\mathbb{Z}/t\mathbb{Z}\times\mathbb{Z}/p\mathbb{Z}.\end{equation}
\end{lem}

\begin{proof} Let $\tilde{H}:=\mathbb{G}_m(C)^r\times\mathbb{G}_a(C)^s\times\mathbb{Z}/t\mathbb{Z}\times\mathbb{Z}/p\mathbb{Z}$ and $\rho:H\twoheadrightarrow \tilde{H}$ be the quotient by $\mathrm{Scal}_n(C)$.

We begin by showing that the connected component of the identity $H^0$ is in the center $Z(H)$, {so} in particular $H^0$ is commutative. For any $g,h\in H^0$, we have that $\rho(ghg^{-1}h^{-1})=1$, and therefore $ghg^{-1}h^{-1}\in\mathrm{Scal}_n(C)$. But since $\mathrm{det}(ghg^{-1}h^{-1})=1$, it follows that $ghg^{-1}h^{-1}$ lies in the cyclic subgroup $\mu_n\subset\mathrm{Scal}_n(C)$ consisting of $n$-th roots of unity. For any fixed $g\in H$, the map $H^0\rightarrow\mu_n:h\mapsto ghg^{-1}h^{-1}$ has image $\{1\}$, since $H^0$ is connected, so $gh=hg$ for any $g\in H$ and $h\in H^0$.

Now we show that $\rho$ restricts to a surjection $H^0\twoheadrightarrow \tilde{H}^0$. The group $\rho(H^0)$ is of finite index in $\tilde{H}$ so $\tilde{H}^0 \subset \rho(H^0)$ {by \cite[Prop.~7.3(b)]{humphreys}}. Since $\rho(H^0)$ is connected, we have $\rho(H^0) \subset \tilde{H}^0$ and so $\rho(H^0) = \tilde{H}^0$.

We claim that $H^0\simeq\mathbb{G}_m(C)^{r+1}\times\mathbb{G}_a(C)^s$. Since $H^0$ is commutative and connected, we know that $H^0\simeq \mathbb{G}_m(C)^{r'}\times\mathbb{G}_a(C)^{s'}$ for some non-negative integers $r'$ and $s'$. We claim that $r'=r+1$ and $s'=s$.  We have $\Scal_n (C)\subset   \mathbb{G}_m(C)^{r'}$, $\rho(\mathbb{G}_m(C)^{r'}) \subset \mathbb{G}_m(C)^{r} $, and $\rho(\mathbb{G}_a(C)^{s'}) \subset \mathbb{G}_a(C)^{s}$.  Comparing dimensions, we have $r' = r+1$ and $s'=s$.

To prove the exact sequence (\ref{tech-ses-1}), let $h \in H$ be chosen so that $\rho(h)$ generates $\ZX/t\ZX$. Since $\rho(h^t) = 1$, we have that $h^t \in \Scal_n(C)$.  Since $\Scal_n(C)$ is divisible, there is some $\bar{h} \in \Scal_n(C)$ such that $\bar{h}^t = h^t$. For $h_1 = h\bar{h}^{-1}$ the group $<h_1>$ generated by $h_1$ is isomorphic to $\ZX/t\ZX$ and $<h_1> \cap H^0 = \{1\}$. Therefore we have an exact sequence 
\[1\longrightarrow H^0\times<h_1>\longrightarrow H \longrightarrow \mathbb{Z}/p\mathbb{Z}\longrightarrow 1,\]
yielding (\ref{tech-ses-1}). A similar argument {produces $h_2 \in H$ of} order $p$ together with an exact sequence 

\[1\longrightarrow H^0\times<h_2>\longrightarrow H \longrightarrow \mathbb{Z}/t\mathbb{Z}\longrightarrow 1,\]
yielding  (\ref{tech-ses-2}). 

In order to prove \eqref{tech-ses-3}, it {would be sufficient} to show that the lifts $h_1$ and $h_2$ chosen above commute. {red Since} $\rho(h_1h_2h_1^{-1}h_2^{-1})=1$, {it follows that} $h_1h_2h_1^{-1}h_2^{-1}\in\mathrm{Scal}_n(C)$. Since $\mathrm{det}(h_1h_2h_1^{-1}h_2^{-1})=1$, it follows that $h_1h_2h_1^{-1}h_2^{-1}=:\alpha\in \mu_n\subset\mathrm{Scal}_n(C)$, the group of $n$-th roots of unity. Since $\alpha h_i=\alpha h_i$ for $i=1,2$, it follows that \begin{gather*}\alpha^t=h_1^{-t}\alpha^t=(h_1^{-1}\alpha)^t=(h_2h_1^{-1}h_2^{-1})^t=h_2h_1^{-t}h_2^{-1}=1
\intertext{and}
\alpha^p=\alpha^ph_2^p=(\alpha h_2)^p=h_1h_2^ph_1^{-1}=1,\end{gather*}
and therefore $\alpha\in\mu_n\cap\mu_t\cap\mu_p=\mu_d$, where $d:=\mathrm{gcd}(n,t,p)$. Hence, the lifts $h_1$ and $h_2$ commute whenever $d=1$.\end{proof}

\begin{thm} \label{p-thm} Suppose that, in case S, Q, or M, the system \begin{equation}\label{eq-p}\sigma(Y)=AY \ \ \text{with} \ A\in\GL_n(C_0(x))\end{equation} is projectively integrable but not integrable. \begin{enumerate}

\item In case S, $G$ is isomorphic to a group of the form \[\mathbb{G}_m(C)\times \mathbb{G}_m(C_0)^r\times\mathbb{Z}/t\mathbb{Z},\] where $r$ is a nonnegative integer and $t$ is a positive integer.

\item In case Q, the connected component of the identity $G^0\simeq W\times \mathbb{G}_m(C_0)^r \times\mathbb{G}_a(C_0)^s$, where {$W$ is either $\mathbb{G}_m(C)$ or $\left\{a\in C^\times \ \middle| \ \delta\bigl(\frac{\delta(a)}{a}\bigr)=0\right\}$,} $r$ and $s$ are nonnegative integers, and $s\leq 1$. There are exact sequences \begin{gather}
\label{pthm-qseq1} 1 \longrightarrow G^0\times \mathbb{Z}/t\mathbb{Z}\longrightarrow G\longrightarrow \mathbb{Z}/p\mathbb{Z} \longrightarrow 1
\intertext{and}
\label{pthm-qseq2} 1 \longrightarrow G^0\times \mathbb{Z}/p\mathbb{Z}\longrightarrow G\longrightarrow \mathbb{Z}/t\mathbb{Z} \longrightarrow 1, \end{gather}
where $p$ and $t$ are positive integers. Moreover, if $\mathrm{gcd}(n,t,p)=1$, then both sequences \eqref{pthm-qseq1} and \eqref{pthm-qseq2} are split, and \begin{equation}\label{pthm-qseq3}G\simeq W\times\mathbb{G}_m(C_0)^r\times\mathbb{G}_a(C_0)^s\times\mathbb{Z}/t\mathbb{Z}\times\mathbb{Z}/p\mathbb{Z}.\end{equation}

\item In case M, the connected component of the identity $G^0\simeq\mathbb{G}_m(C)\times\mathbb{G}_m(C_0)^r\times\mathbb{G}_a(C_0)^s$, where $r$ and $s$ are nonnegative integers and $s\leq 1$. There are exact sequences \begin{gather}
\label{pthm-mseq1} 1 \longrightarrow G^0\times \mathbb{Z}/t\mathbb{Z}\longrightarrow G\longrightarrow \mathbb{Z}/p\mathbb{Z} \longrightarrow 1
\intertext{and}
\label{pthm-mseq2} 1 \longrightarrow G^0\times \mathbb{Z}/p\mathbb{Z}\longrightarrow G\longrightarrow \mathbb{Z}/t\mathbb{Z} \longrightarrow 1, \end{gather}
where $t$ and $p$ are positive integers. Moreover, if $\mathrm{gcd}(n,t,p)=1$, then both sequences \eqref{pthm-mseq1} and \eqref{pthm-mseq2} are split, and \[G\simeq \mathbb{G}_m(C)\times\mathbb{G}_m(C_0)^r\times\mathbb{G}_a(C_0)^s\times\mathbb{Z}/t\mathbb{Z}\times\mathbb{Z}/p\mathbb{Z}.\]
\end{enumerate}
\end{thm}

\begin{proof} (2) If $H\subset\mathrm{GL}_n(C)$ is the $\sigma$-Galois group of \eqref{eq-p} over $k$, it follows from Corollary~\ref{nonint-scal-cor} that $\mathrm{Scal}_n(C)\subset H$. By Proposition~\ref{piso-seq-prop}, $H/\mathrm{Scal}_n(C)\simeq\tilde{H}$, where $\tilde{H}$ is the $\sigma$-Galois group of an integrable system $\sigma(Y)=\tilde{A}Y$ with $\tilde{A}\in\mathrm{G}_{n^n}(C_0(x))$. By Proposition~\ref{intprop} and Theorem~\ref{thm4}, $\tilde{H}\simeq\mathbb{G}_m(C)^r\times\mathbb{G}_a(C)^s\times\mathbb{Z}/t\mathbb{Z}\times\mathbb{Z}/p\mathbb{Z}$, where $r$ and $s$ are nonnegative integers, $s\leq 1$, and $t$ and $p$ are positive integers. 

We have just shown that $H$ satisfies the hypotheses of Lemma~\ref{tech-lem}, which we remark remains valid after replacing $C$ with {$C_0$}. In the proof of Lemma~\ref{tech-lem}, we found partial sections \[\psi_0:\tilde{H}^0(C_0)\hookrightarrow H^0(C_0); \ \ \psi_1:\tilde{H}^0(C_0)\times\mathbb{Z}/t\mathbb{Z}\hookrightarrow H(C_0); \ \ \text{and} \ \ \psi_2:\tilde{H}^0(C_0)\times\mathbb{Z}/p\mathbb{Z} \hookrightarrow H(C_0)\] to the surjection $\rho:H\twoheadrightarrow \tilde{H}$ yielding the isomorphism $H^0(C_0)\simeq\mathrm{Scal}_n(C_0)\times\tilde{H}^0(C_0)$ as well as (the restrictions to $C_0$-points of) the exact sequences \eqref{tech-ses-1} and \eqref{tech-ses-2}. Under the supplemental hypothesis that $\mathrm{gcd}(n,t,p)=1$, we also showed that there is a section $\psi_3:\tilde{H}(C_0)\hookrightarrow H(C_0)$ yielding the isomorphism \eqref{tech-ses-3}.

From Proposition~\ref{piso-seq-prop} and Proposition~\ref{intprop}, we have an exact sequence \begin{equation}\label{pthm-gseq}1\longrightarrow G\cap\mathrm{Scal}_n(C)\longrightarrow G\longrightarrow \tilde{H}(C_0)\longrightarrow 1.\end{equation} Hence, to prove Theorem~\ref{p-thm}(2)---i.e., that $G^0\simeq(G\cap\mathrm{Scal}_n(C))\times\tilde{H}^0(C_0)$; the existence of the exact sequences \eqref{pthm-qseq1} and \eqref{pthm-qseq2}; and that if $\mathrm{gcd}(n,t,p)=1$ then \eqref{pthm-qseq3} holds---it suffices to show that the image of any of these partial sections $\psi$ is contained in $G$. To see this, let $\tilde{h}\in\tilde{H}(C_0)$ and $\psi(\tilde{h})=:h\in H(C_0)$. We wish to show that $h\in G$. In any case, it follows from \eqref{pthm-gseq} that there exists $g\in G$ such that $\rho(g)=\tilde{h}$. It follows that $\rho(gh^{-1})=1$ and therefore $g=\alpha h$ for some $\alpha\in\mathrm{Scal}_n(C)$. Since $h\in H(C_0)$ and $\alpha h=h\alpha$, {it follows that} $\delta(\alpha h)(\alpha h)^{-1}=\delta(\alpha)\alpha^{-1}$. Since $\alpha\in\mathrm{Scal}_n(C)$, {we have} $\delta(\alpha)\alpha^{-1}=\frac{1}{n}(\delta(\mathrm{det}(\alpha))/\mathrm{det}(\alpha))\cdot I_n$. On the other hand, since $G\subset \mathrm{GL}_n(C_0)\cdot\mathrm{Scal}_n(C)$, we know that $\delta(g)g^{-1}=\frac{1}{n}(\delta(\mathrm{det}(g))/\mathrm{det}(g))\cdot I_n$. It follows from Lemma~\ref{det} that $\alpha\in G\cap\mathrm{Scal}_n(C)$, and therefore $h\in G$, as we wanted to show.

(1)~and (3)~The proof of case Q also works in cases S and M, {where it also} follows from Lemma~\ref{det} that $G\cap\mathrm{Scal}_n(C)\simeq\mathbb{G}_m(C)$. {Moreover, in} case S it follows from Theorem~\ref{thm4}(1) that $s=0$ and $p=1$.\end{proof}

\subsection{Example}

The following example shows that Theorem~\ref{p-thm}(2) is optimal, in the sense that sequences \eqref{pthm-qseq1} and \eqref{pthm-qseq2} are not always split.  Working in case q, we claim that the $\sigma$-Galois group $G$ for \begin{equation}\label{eg-eq} \sigma^2(y)=q^{1/2}xy \quad \text{or equivalently}\quad \sigma(Y)=\begin{pmatrix} 0 & 1 \\ q^{1/2}x & 0\end{pmatrix} Y\end{equation} is $G=Q_8\cdot\mathbb{G}_m(C)=D_8\cdot\mathbb{G}_m(C)$,  a projectively $\delta$-constant group. To be precise, we claim that \begin{equation}\label{g}G=\left\{\begin{pmatrix} \alpha & 0 \\ 0 & \alpha\end{pmatrix}, \ \begin{pmatrix} \alpha & 0 \\ 0 & -\alpha\end{pmatrix}, \ \begin{pmatrix} 0 & \alpha \\ \alpha & 0 \end{pmatrix}, \ \begin{pmatrix} 0 & \alpha \\ -\alpha & 0 \end{pmatrix} \ \middle| \ \alpha\in C^\times\right\}.\end{equation}

Fix once and for all square roots $q^{1/2}$ and $x^{1/2}$ of $q$ and $x$, respectively, and extend $\sigma$ to $k_2:=k(x^{1/2})$ by $\sigma(x^{1/2})=q^{1/2}x^{1/2}$. Considering  \eqref{eg-eq} as an equation over $k_2$, note that \eqref{eg-eq}  is gauge equivalent to \begin{equation}\label{gauge} \sigma(Z)=\begin{pmatrix} x^{1/2} & 0 \\ 0 & -x^{1/2}\end{pmatrix}Z,\end{equation} via the gauge transformation given by \[T:=\begin{pmatrix}1 & x^{-1/2}\\ 1 & -x^{-1/2}\end{pmatrix}.\] Let $Z:=\mathrm{diag}(z_1, z_2 )$ denote a fundamental solution matrix for \eqref{gauge} over $k_2$, so that $\sigma(z_1)=x^{1/2}z_1$ and $\sigma(z_2)=-x^{1/2}z_2$, each $z_i$ is invertible, and $z_1$ and $z_2$ are {$C$-}linearly independent.

If we let $y_1:=z_1+z_2$ and $y_2:=z_1-z_2$ then \begin{equation}\label{yfundsol} Y:=\begin{pmatrix} y_1 & y_2 \\ \sigma(y_1) & \sigma(y_2)\end{pmatrix}\end{equation} is a fundamental solution  matrix for \eqref{eg-eq} over $k_2$ that satisfies \begin{equation} \label{basis}\sigma(y_1)=x^{1/2}y_2 \quad\text{and} \quad \sigma(y_2)=x^{1/2}y_1.\end{equation} Let $S:=k_2[Y,\mathrm{det}(Y)^{-1}]$ denote the corresponding Picard-Vessiot-ring for \eqref{eq} over $k_2$. Since \[\frac{y_1^2-y_2^2}{y_1\sigma(y_2)-y_2\sigma(y_1)}=\frac{y_1^2-y_2^2}{x^{1/2}(y_1^2-y_2^2)}=x^{-1/2},\] we have that $k_2\subset k[Y,\mathrm{det}(Y)^{-1}]=S$ is also a Picard-Vessiot ring for \eqref{eq} over $k$. Let $G:=\mathrm{Gal}_\sigma(S/k)$.

We claim that $y_1y_2=0$. By the results of \cite{hendriks:1997}, the choice of fundamental solution matrix $Y$ as in \eqref{yfundsol} identifies $G$ with a subgroup of \[\left\{\begin{pmatrix}  \alpha & 0 \\ 0 & \lambda  \end{pmatrix} \ \middle| \ \alpha,\lambda\in C^\times\right\} \cup \left\{\begin{pmatrix}  0 & \beta \\ \epsilon & 0  \end{pmatrix} \ \middle| \ \beta,\epsilon\in C^\times\right\}.\] If $T_\gamma\in\mathrm{GL}_2(C)$ is the matrix associated to $\gamma\in G$ we see that $\gamma(y_1y_2)=\pm\mathrm{det}(T_\gamma)y_1y_2$ for every $\gamma\in G$. On the other hand, if we let $\omega:=y_1\sigma(y_2)-y_2\sigma(y_1)=\mathrm{det}(Y)$ we see that $\gamma(\omega)=\mathrm{det}(T_\gamma)\omega$ for every $\gamma\in G$. Therefore $\gamma(y_1^2y_2^2/\omega^2)=y_1^2y_2^2/\omega^2$ for every $\gamma\in G$, which implies that $y_1^2y_2^2=f\omega^2$ for some $f\in k$, by the Galois correspondence. We proceed by contradiction, as in \cite[\S8.3]{Arreche15b}. Note that if $y_1y_2\neq 0$, then $f\neq 0$ is invertible. On the other hand, $\sigma^2(y_1y_2)=qx^2y_1y_2$, by \eqref{basis}. Since $\sigma(\omega)=-q^{1/2}x\omega$, we have that $\sigma^2(\omega)=q^{3/2}x^2\omega$. It follows that $\sigma^2(y_1y_2/\omega)=q^{-1/2}y_1y_2/\omega$, and therefore $\sigma^2(y_1^2y_2^2/\omega^2)=q^{-1}y_1^2y_2^2/\omega^2$. It follows that \[\frac{\sigma(f\sigma(f))}{f\sigma(f)}=\frac{\sigma^2(f)}{f}=\frac{\sigma^2(y_1^2y_2^2/\omega^2)}{y_1^2y_2^2/\omega^2}=q^{-1},\] and therefore there exists $c\in C^\times$ such that $f\sigma(f)=cx^{-1}$. But this is impossible, because $f\sigma(f)$ has even order of vanishing at $0$. This proves that $y_1y_2=0$, as we wanted to show.

There is a decomposition $S=S_0\oplus \dots\oplus S_{t-1}$, where each $S_i$ is a domain and a $\sigma^t$-Picard-Vessiot ring over $k$, considered as a $\sigma^t$-field, for $\sigma^t(Y)=\sigma^{t-1}(A)\dots\sigma(A)AY$, where $A:=\left(\begin{smallmatrix}0 &1 \\ q^{1/2}x & 0\end{smallmatrix}\right)$ and $\sigma:S_{i \pmod{t}}\rightarrow S_{i+1\pmod{t}}$ is an isomorphism of $\sigma^t$-rings . We claim that $t=2$. In general, $t$ is the largest integer such that there exists $z\in S$ with $\sigma(z)=\zeta_tz$, where $\zeta_t$ is a primitive $t$-th root of unity. Since $\sigma(z_1/z_2)=-z_1/z_2$, we have that $t$ is even, and therefore $S=S_{(0)}\oplus S_{(1)}$, where each of \[S_{(0)}:=\bigoplus_{i \ \text{even}}S_i\qquad\text{and}\qquad S_{(1)}:=\bigoplus_{i \ \text{odd}}S_i\] is a $\sigma^2$-Picard-Vessiot ring over $k$ for\begin{equation} \label{eq2} \sigma^2(Y)=\begin{pmatrix} q^{1/2}x & 0 \\ 0 & q^{3/2}x\end{pmatrix}Y,\end{equation} where we consider $k$ as a $\sigma^2$-field.

Let us consider \eqref{eq2} as an equation over the underlying $\sigma^2$-field of $k_2$, and observe that \begin{equation}\label{basis2}\begin{pmatrix}y_1 & 0\\ 0 & x^{1/2}y_1\end{pmatrix}=\begin{pmatrix} w_1 & 0 \\ 0 & w_2\end{pmatrix}\end{equation} is a fundamental solution matrix for \eqref{eq2} over $k_2$. We see that $k_2[y_1,y_1^{-1}]$ is a $\sigma^2$-Picard-Vessiot ring for \eqref{eq2} over $k_2$, and that $k_2\subset k[w_1,w_2,(w_1w_2)^{-1}] =k_2[y_1,y_1^{-1}]$ is a $\sigma^2$-Picard-Vessiot ring for \eqref{eq2} over the underlying $\sigma^2$-field of $k$. Using the algorithm of \cite{hendriks:1997}, we see that  $S_{(0)}$ and $S_{(1)}$ are domains, which implies that $t=2$, $S_{(0)}=S_0$, and $S_{(1)}=S_1$. 

It follows from \eqref{basis}  and the fact that $y_1y_2=0$ that we may take $S_0=k_2[y_1,y_1^{-1}]$ and $S_1=k_2[y_2,y_2^{-1}]$, which we emphasize are actually $\sigma^2$-Picard-Vessiot rings for \eqref{eq2} over $k$, not just over $k_2$. Letting $G_2:=\mathrm{Gal}_{\sigma^2}(S_0/k)$, we have that the embedding $G_2\hookrightarrow\mathrm{GL}_2(C)$ corresponding to the solution matrix \eqref{basis2} is given by \begin{equation}\label{g2}G_2=\left\{\begin{pmatrix}\alpha & 0 \\ 0 & \alpha\end{pmatrix}, \ \begin{pmatrix} \alpha & 0 \\ 0 & -\alpha\end{pmatrix} \ \middle| \ \alpha\in C^\times\right\},\end{equation} where $\gamma(w_1)=\alpha w_1$ and $\gamma(w_2)=\pm\alpha w_2$ for every $\gamma\in G_2$ {and} the sign is determined by $\gamma(\frac{w_2}{w_1})=\gamma(x^{1/2})=\pm x^{1/2}.$ 

There is also an exact sequence \begin{equation}\label{sequence}1\rightarrow G_2\rightarrow G\rightarrow\mathbb{Z}/2\mathbb{Z}\rightarrow 1,\end{equation} where the map $G_2\rightarrow G:\gamma\mapsto\tilde{\gamma}$ is determined by the rule $\tilde{\gamma}(y_1)=\gamma(y_1)$ and $\tilde{\gamma}(y_2)=\sigma\gamma\sigma^{-1}(y_2).$ It follows from \eqref{basis} that \[\sigma\gamma\sigma^{-1}(y_2)=\sigma\gamma(q^{1/2}x^{-1/2}y_1)=\sigma(q^{1/2}(\pm x^{-1/2})(\alpha y_1))=q^{1/2}(\pm q^{-1/2}x^{-1/2})(\alpha x^{1/2} y_2)=\pm\alpha y_2.\] Therefore the image of $G_2\hookrightarrow G:\gamma\mapsto\tilde{\gamma}$ is also given by \eqref{g2}.

Finally, it is clear that the $k$-linear map defined by $y_1\mapsto y_2$ and $y_2\mapsto y_1$ is a $\sigma$-automorphism of $S$ over $k$. Therefore $\left(\begin{smallmatrix} 0 & 1 \\ 1 & 0\end{smallmatrix}\right)\in G$ lifts the generator of $\mathbb{Z}/2\mathbb{Z}$ in \eqref{sequence}, and $G$ is given by \eqref{g}, as claimed.

\section{Hypertranscendence}\label{hyper-sec} 

In \cite{DHR15} and \cite{DHR16}, the authors  develop criteria to guarantee that the $\sd$-Galois group of (\ref{introeq}) is large when $\sigma$ is a $q$-dilation or a Mahler operator.  {They prove the validity of their criteria} by showing that certain {linear differential algebraic} groups cannot occur as $\sd$-Galois groups of integrable or projectively integrable equations. In their proof, they did not classify the $\sd$-Galois groups of such equations. Using this classification, we are now able to reexamine and extend their criteria.  We  begin with some preliminary results concerning Galois groups.

\begin{lem}\label{semilem0} Let $C,C_0, k, \sigma$ and $\delta$ be as in one of the cases S, Q, or M. If the identity component of the $\sigma$-Galois group $H(C)$  of 
\[\sigma(Y) = AY, \, A\in \mathrm{GL}_n(C_0(x))\]
over $k$ is a  semisimple linear algebraic group then the $\sigma\delta$-Galois group is also $H(C)$. 
\end{lem}
\begin{proof} Proposition~\ref{propdense} implies that the $\sigma\delta$-Galois group $G$ of this equation is Zariski-dense in {$H$}. In particular, the Kolchin-connected component $G^0$ of $G$ is Zariski-dense in the Zariski-connected component $H^0$ of $H$.   Let $H^0(C)= \prod_{i=1}^dH_i(C), \text{where the  } H_i$ are the minimal closed connected normal subgroups of  $H^0$ as in \cite[Theorem 27.5]{humphreys} ({we emphasize that} this is not a direct product, just a product). 
Cassidy \cite[Thms.~19 and 20]{Cassidy89} has shown that a Zariski-dense differential algebraic subgroup $G^0$ of $H^0$  is  conjugate to a group of the form $\prod_{i=1}^dG_i$ where for each $i$ {either} $G_i = H_i(C)$ or $G_i =H_i(C_0)$. We write this product as
\[G^0 = \prod_{i \in I} G_i\prod_{i\in J} G_i,\]
where $G_i = H_i(C)$ for $i\in I$ and $G_i = H_i(C_0)$ for $i \in J$. 

We claim that $ \prod_{i \in I} G_i$ is normal in $G$. For any $g \in H$, {we have that} $gG_i(C)g^{-1}$ is again a minimal normal simple subgroup of $H^0$, and so 
$gG_i(C)g^{-1} = G_j(C)$ for some $j$.   We have {that} $G'= G_j\cap \prod_{\ell\neq j}{G_\ell}$ is finite \cite[Thm.~27.5]{humphreys}.   Therefore if $j \in J$ then $gG_ig^{-1}$ lies in $G' \cdot G_j(C_0)$.  This latter group has differential type $0$, contradicting the fact that $G_j(C)$ has differential type $1$. Therefore $ \prod_{i \in I} G_i$ is normal in $G$.

We now claim that $J$ is empty.  The Galois correspondence implies that there is a difference equation $\sigma(Y) = A_0Y, \, A_0\in \mathrm{GL}_m(k)$ with $\sigma$-Galois group $\tilde{H}=H/\prod_{i\in I}H_i$ and $\sd$-Galois group $\tilde{G} = G/\prod_{i \in I} G_i$. We may assume without loss of generality that the subgroup $\prod_{i\in I} H_i$ is defined over $C_0$. The Tannakian theory implies that the difference module associated to $\sigma(Y)=A_0Y$ can be obtained from the difference module associated to $\sigma(Y)=AY$ using only constructions of linear algebra over $C_0(x)$. Therefore, we may assume that $A_0 \in \GL_m(C_0(x))$.

The identity component $\tilde{G}^0$ of $\tilde{G}$ is the quotient of a connected $\delta$-constant semisimple group and so is $\delta$-constant as well (cf.\cite[Cor.~1]{Ov08} or \cite[Thm.~4.6]{GoOv14}). Since $G^0$ is $\delta$-constant, we have that $G$ is $\delta$-constant as well \cite[Prop.~2.48]{HMO15}, and so $\sigma(Y) = A_0Y$ is integrable.  Theorem~\ref{thm4} implies that  the $\sigma\delta$-Galois group of an integrable equation cannot be of this form. Therefore $J$ is empty and the $\sigma\delta$-Galois group is also $H(C)$. \end{proof}

We can now prove the following. If $H$ is a linear algebraic group we denote the derived subgroup of the identity component of $H$ by $H^{0,der}$.

\begin{thm}\label{thmSQM} Let $C$, $C_0$, $k$, $\sigma$, and $\delta$ be as in one of the cases S, Q, or M. Let $H$ be the $\sigma$-Galois group over $C_0(x)$ of

\begin{equation}\label{DHReq}\sigma(Y) = AY, \ \ A \in \GL_n(C_0(x)).\end{equation}

Assume that $H^0$ is a reductive group. Then, the $\sd$-Galois group of this equation over $k$ is a subgroup of $H(C)$ containing $H^{0,der}(C)$.
\end{thm}

\begin{proof} Let $K$ be the total quotient ring of a $\sd$-Picard-Vessiot extension $R = k\{Z,\frac{1}{\det Z}\}_{\delta}$ of $k$ corresponding to (\ref{DHReq}), and let $G$ be its $\sd$-Galois group over $k$. Then $S = k[Z,\frac{1}{\det Z}]$ is a $\sigma$-Picard-Vessiot extension for this equation, and its total quotient ring $F$ is a subring of $K$ \cite[Cor.~20]{FSW10}. Since $H^0$ is reductive, we have that $H^0 = Z(H^0)\cdot H^{0,der}$, where $Z(H^0)$ is the center of $H^0$ and $H^{0,der}$ is semisimple.  

Let  $\tilde{G} = H^{0,der}\cap G$. We claim that $\tilde{G}$ is Zariski-dense in   {$H^{0,der}$}. To see this, note that $G^0 = G\cap H^0 $ is Zariski-dense in $H^0 $, and therefore $G^0 \times G^0 $ is Zariski-dense in $H^0 \times H^0 $.  By \cite[Thm.]{Ree64}, the map $\phi:(a,b)\mapsto aba^{-1}b^{-1}$ maps $H^0 \times H^0 $ surjectively onto $H^{0,der} $. Therefore $\tilde{G}  \supset  \phi( G^0 \times G^0 )$ is Zariski-dense in $H^{0,der} $.

We will argue by contradiction that $\tilde{G}  = H^{0,der} $. Since  $H^{0,der}$ is semisimple, we may write $H^{0,der}  = \prod_{i=1}^dH_i $ as in Lemma~\ref{semilem0}.  If $\tilde{G} $ is a proper subgroup of $H^{0,der} $, then we may assume that $\tilde{G}  =\prod_{i=1}^dG_i$, where for each $i$  either $G_i  = H_i $ or  else $G_i $ is conjugate to $H_i(C_0)$.  Assuming that $\tilde{G}  \neq H^{0,der} $, we have that some $G_i  = H_i(C_0)$.

 Since $Z(H^0)$ is normal in $H$, we can form $H' = H/Z(H^0)$. By the Galois correspondence, $H'$ is the $\sigma$-Galois group for a difference equation \begin{equation}\label{eq001}
\sigma(Y) = A'Y, \ \ A' \in \GL_m(C_0(x)).
\end{equation} To see this, we proceed as in the proof of Lemma~\ref{semilem0}: since $Z(H^0)$ is defined over $C_0$, a Tannakian argument shows that the difference module associated to \eqref{eq001} is obtained from the difference module associated to \eqref{DHReq} using only constructions of linear algebra over $C_0(x)$. Moreover, there is a $\sigma$-Picard-Vessiot ring $S'=k[Z',\frac{1}{\mathrm{det} \ Z'}]\subset F^{Z(H^0)}$, where $Z'$ is a fundamental solution for \eqref{eq001}, and we see that $Z(H^0)=\{\gamma\in H \ | \ \gamma(Z')=Z'\}$. It follows that the $\sigma\delta$-Galois group of the $\sigma\delta$-Picard-Vessiot ring $R'=k\{Z',\frac{1}{\mathrm{det} \ Z'}\}_\delta$ for \eqref{eq001} is precisely $G'=G/G\cap Z(H^0)$. Since $H'$ is a finite extension of the semisimple group $H^{0,der}/Z(H^0)\cap H^{0,der}$, Lemma~\ref{semilem0} implies that $G'=H'$. Since $G'$ is a finite extension of $\tilde{G}/\tilde{G}\cap Z(H^0)$, it is not possible for $\tilde{G}$ to be a proper subgroup of $H^{0,der}$. \end{proof}

Theorem~\ref{thmSQM} allows us to reprove and generalize the results of \cite{DHR15} and \cite{DHR16}. In \cite{DHR15}, the authors consider case M and show in Theorem 3.5 that if $H$ is a subgroup of $\GL_n(C)$ containing $\SL_n(C)$ and the $\sd$-Galois group of $\sigma(y) = \det(A)y$ is a subgroup of $C_0^\times$ then the $\sd$-Galois group of (\ref{DHReq})  is a subgroup of $C_0^\times\SL_n(C)$ containing $\SL_n(C)$.  Theorem~\ref{thmSQM} allows us to remove the restriction on the $\sd$-Galois group of $\sigma(y) = \det(A)y$ and {also show} that $C^\times\SL_n(C_0)$ cannot be a $\sd$-Galois group over $k$  for a difference system defined over $C_0(x)$ in cases S, Q, or M.

 In Part I of \cite{DHR16}, the authors consider case Q and show in Theorem 3.1 that if $H^{0,der}$ is an irreducible almost simple algebraic subgroup of $\SL_n(C)$ then the $\sd$-Galois group contains $H^{0,der}$. This result easily follows from, and is generalized by, Theorem~\ref{thmSQM}.  Furthermore, this latter theorem now extends these results to cases S and M.
 
 The hypertranscendence results of \cite{DHR15} and \cite{DHR16} (e.g., the hypertranscendence of certain generalized hypergeometric series and the generating functions of the Baum-Sweet and Rudin-Shapiro sequences)  follow as indicated in these papers.  {Similarly}, Theorem~\ref{thmSQM} can be used to show hypertranscendence results in case S.

\begin{eg} Let $C$, $C_0$, $k$, $\sigma$, and $\delta$ be as in  case S.  In \cite[Lem.~3.9]{PuSi}, the authors show that the equation
\[Y(x+1) = \begin{pmatrix}0 & -1\\ 1 & a\end{pmatrix} Y(x),\]
 where $a \in C_0[x]$ {and} $a(0) = 0$, has $\sigma$-Galois group $\SL_2(C_0)$ over $C_0(x)$.  The above results show that the $\sd$-Galois group is $\SL_2(C)$ and so the differential transcendence degree of the associated  $\sd$-Picard-Vessiot extension is $3$.  Therefore if 
\[Z = \begin{pmatrix} z_{1,1} & z_{1,2} \\z_{2,1} & z_{2,2}\end{pmatrix}\]
is a fundamental solution matrix in any $\sd$-Picard-Vessiot extension {then} any three of the entries of $Z$ are differentially independent.

Furthermore, {in} \cite[Lem.~3.9]{PuSi} the authors show that for $A = {\rm diag}(A_1, \ldots , A_m)$, where 
\[A_i = \begin{pmatrix} 0 & -1 \\ 1 & x^i \end{pmatrix},\]
the equation $Y(x+1) = A(x) y(x)$ has $\sigma$-Galois group $\SL_2(C_0)^m$ over $C_0(x)$. Therefore the differential transcendence degree of the associated $\sd$-Picard-Vessiot extension is $3m$. This implies that {if} $K$ is a $\sd$-Picard-Vessiot extension of $C_0(x)$ {and} $Z_\ell = (z_{i,j}^{\ell}) \in \SL_2(K)$ satisfies {$Z_\ell(x+1) = A_\ell(x) Z_\ell(x)$ for each $\ell = 1, \ldots , m$} then the elements \[\{z_{1,1}^{1}, z_{1,2}^{1}, z_{2,1}^{1}, z_{1,1}^{2}, z_{1,2}^{2}, z_{2,1}^{2}, \ldots , z^m_{1,1},z^m_{1,2},z^{m}_{2,1}\}\] are differentially independent.
\end{eg}

\newcommand{\SortNoop}[1]{}\def\cprime{$'$} \def\cprime{$'$} \def\cprime{$'$}
  \def\cprime{$'$}


\begin{thebibliography}{10}

\bibitem{Arreche15b}
C.E. Arreche.
\newblock Computation of the difference-differential {Galois} group and differential relations among solutions for a second-order linear difference equation.
\newblock arXiv:1606.07109 [math.AC]; to appear in {\it Commun.~Contemp.~Math.},
  2016.

\bibitem{B94}
P-G. Becker.
\newblock {$k$}-regular power series and {M}ahler-type functional equations.
\newblock {\em J. Number Theory}, 49(3):269--286, 1994.

\bibitem{Bez94}
J-P. B{\'e}zivin.
\newblock Sur une classe d'\'equations fonctionnelles non lin\'eaires.
\newblock {\em Funkcial. Ekvac.}, 37(2):263--271, 1994.

\bibitem{Cassidy72}
P.~J. Cassidy.
\newblock Differential algebraic groups.
\newblock {\em Amer. J. Math.}, 94:891--954, 1972.

\bibitem{Cassidy89}
P.~J. Cassidy.
\newblock The classification of the semisimple differential algebraic groups
  and the linear semisimple differential algebraic {L}ie algebras.
\newblock {\em J. Algebra}, 121(1):169--238, 1989.

\bibitem{CaSi}
P.~J. Cassidy and M.~F. Singer.
\newblock Galois theory of parameterized differential equations and linear
  differential algebraic groups.
\newblock In {\em Differential equations and quantum groups}, volume~9 of {\em
  IRMA Lect. Math. Theor. Phys.}, pages 113--155. Eur. Math. Soc., Z\"urich,
  2007.

\bibitem{CHS}
Z.~Chatzidakis, C.~Hardouin, and M.~F. Singer.
\newblock On the definitions of difference {G}alois groups.
\newblock In {\em Model theory with applications to algebra and analysis.
  {V}ol. 1}, volume 349 of {\em London Math. Soc. Lecture Note Ser.}, pages
  73--109. Cambridge Univ. Press, Cambridge, 2008.

\bibitem{DiHa12}
L.~Di~Vizio and C.~Hardouin.
\newblock Descent for differential {G}alois theory of difference equations:
  confluence and {$q$}-dependence.
\newblock {\em Pacific J. Math.}, 256(1):79--104, 2012.

\bibitem{DHR15}
T.~Dreyfus, C.~Hardouin, and J.~Roques.
\newblock Hypertranscendence of solutions of {Mahler} equations.
\newblock arXiv:1507.03361 [math.AC]; to appear in {\it J.~Eur.~Math.~Soc.},
  2015.

\bibitem{DHR16}
T.~Dreyfus, C.~Hardouin, and J.~Roques.
\newblock Functional relations of solutions of $q$-difference equations.
\newblock arXiv:1603.06771 [math.NT], 2016.

\bibitem{FSW10}
R.~Feng, M.~F. Singer, and M.~Wu.
\newblock Liouvillian solutions of linear difference-differential equations.
\newblock {\em J. Symbolic Comput.}, 45(3):287--305, 2010.

\bibitem{GoOv14}
S. Gorchinskiy and A. Ovchinnikov.
\newblock Isomonodromic differential equations and differential categories.
\newblock {\em J. Math. Pures Appl. (9)}, 102(1):48--78, 2014.

\bibitem{HMO15}
C.~Hardouin, A.~Minchenko, and A.~Ovchinnikov.
\newblock Calculating {G}alois groups of differential equations with parameters
  and hypertranscendence.
\newblock arXiv:1505.07068v2 [math.AC], 2015.

\bibitem{HaSi08}
C.~Hardouin and M.F. Singer.
\newblock Differential {G}alois theory of linear difference equations.
\newblock {\em Math. Ann.}, 342(2):333--377, 2008.
\newblock Erratum in {\it Math. Ann.} (20011) 350:243-244, DOI
  10.1007/s00208-010-0551-1.

\bibitem{hendriks:1997}
P.~A. Hendriks.
\newblock {An algorithm for computing the standard form for second order linear
  $q$-difference equations}.
\newblock {\em J. Pure Appl. Algebra}, 117--118:331--352, 1997.
\newblock
  \href{http://dx.doi.org/10.1016/S0022-4049(97)00017-0}{\nolinkurl{doi:10.1016/S0022-4049(97)00017-0}}.

\bibitem{hendriks:1998}
P.~A. Hendriks.
\newblock {An algorithm determining the difference Galois group of second order
  linear difference equations}.
\newblock {\em J. Symbolic Comput.}, 26(4):445--461, 1998.
\newblock
  \href{http://dx.doi.org/10.1006/jsco.1998.0223}{\nolinkurl{doi:10.1006/jsco.1998.0223}}.

\bibitem{humphreys}
J.~E. Humphreys.
\newblock {\em Linear algebraic groups}.
\newblock Springer-Verlag, New York-Heidelberg, 1975.
\newblock Graduate Texts in Mathematics, No. 21.

\bibitem{OW15}
A.~Ovchinnikov and M.~Wibmer.
\newblock {$\sigma$}-{G}alois theory of linear difference equations.
\newblock {\em Int. Math. Res. Not. IMRN}, (12):3962--4018, 2015.

\bibitem{Ov08}
A. Ovchinnikov.
\newblock Tannakian approach to linear differential algebraic groups.
\newblock {\em Transform. Groups}, 13(2):413--446, 2008.

\bibitem{PuSi}
M.~{\SortNoop{Put}}van~der Put and M.F. Singer.
\newblock {\em Galois theory of difference equations}, volume 1666 of {\em
  Lecture Notes in Mathematics}.
\newblock Springer-Verlag, Berlin, 1997.

\bibitem{Ramis92}
J.-P. Ramis.
\newblock About the growth of entire functions solutions of linear algebraic
  {$q$}-difference equations.
\newblock {\em Ann. Fac. Sci. Toulouse Math. (6)}, 1(1):53--94, 1992.

\bibitem{Ree64}
R. Ree.
\newblock Commutators in semi-simple algebraic groups.
\newblock {\em Proc. Amer. Math. Soc.}, 15:457--460, 1964.

\bibitem{SS16}
R.~Sch\"afke and M.F. Singer.
\newblock Consistent systems of linear differential and difference equations.
\newblock arXiv:1605.02616 [math.CA], 2016.

\bibitem{wibmer10}
M.~Wibmer.
\newblock {\em {Geometric Difference Galois Theory}}.
\newblock PhD thesis, Heidelberg, 2010.

\bibitem{Wib12}
M.~Wibmer.
\newblock Existence of {$\partial$}-parameterized {P}icard-{V}essiot extensions
  over fields with algebraically closed constants.
\newblock {\em J. Algebra}, 361:163--171, 2012.

\end{thebibliography}
\end{document}